\documentclass[11pt,a4paper]{amsart}
\usepackage{amssymb,amsmath}
\usepackage{graphicx}
\usepackage{bm,bbm}
\usepackage{color}
\usepackage{tikz}
\usetikzlibrary{snakes}
\usetikzlibrary{patterns}
\usepackage{pgfplots}
\usepackage{enumitem}
\usepackage{array,blkarray}
\usepackage{bigdelim}
\usepackage{listings} 
\usepackage[hidelinks]{hyperref}
\usepackage{subcaption}
\usepackage{mwe}
\usepackage[normalem]{ulem}
\newtheorem{theorem}{Theorem}

\newtheorem{corollary}[theorem]{Corollary}
\theoremstyle{definition}

\theoremstyle{lemma}
\newtheorem{lemma}[theorem]{Lemma}

\theoremstyle{remark}
\newtheorem{remark}[theorem]{Remark}

\numberwithin{theorem}{section}
\numberwithin{equation}{section}
\numberwithin{table}{section}
\numberwithin{figure}{section}
\usepackage{bbold}
\newcommand{\ddiv}{\operatorname{div}}
\newcommand{\tri}{\mathcal{T}}
\newcommand{\R}{\mathbb{R}}
\newcommand{\id}{\mathbb{1}}
\newcommand{\tr}{\mathrm{tr}}

\newcommand{\cor}{\mathcal{C}}
\newcommand{\Ker}{\mathrm{Ker}}
\newcommand{\opnorm}[1]{{\left\vert\kern-0.25ex\left\vert\kern-0.25ex\left\vert #1 
		\right\vert\kern-0.25ex\right\vert\kern-0.25ex\right\vert}}
\newcommand{\calG}{\mathcal{G}}

\newcommand\dx{\,\text{d}x}

\newcommand{\DtoN}{\mathfrak{L}}
\newcommand{\dist}{\texttt{\upshape{dist}}}
\newcommand{\MS}{\mathcal{M}}
\newcommand{\SM}{\mathcal{S}}
\newcommand{\AC}{\mathcal{A}}
\newcommand{\VHo}{V_H^0}

\newcommand{\VH}{V_H}
\newcommand{\Vh}{V_h}
\newcommand{\V}{V}
\newcommand{\Vo}{V^0}
\newcommand{\Ri}{R}
\newcommand{\RiH}{R_H}
\newcommand{\X}{X}
\newcommand{\XH}{X_H}
\newcommand{\Eb}{E^b}
\newcommand{\EbH}{E^b_H}
\definecolor{color2}{RGB}{255, 126, 126}
\definecolor{color3}{RGB}{0, 100, 0}
\definecolor{color4}{RGB}{30, 140,200}
\definecolor{color1}{RGB}{176, 226, 255}
\definecolor{myOrange}{RGB}{225,92,22} 

\usetikzlibrary{shapes.misc}

\tikzset{cross/.style={cross out, draw=black, minimum size=2*(#1-\pgflinewidth), inner sep=0pt, outer sep=0pt},
	cross/.default={0.6ex}}

\begin{document}
\title[Reconstruction of effective models from low-resolution data]{Reconstruction of quasi-local numerical effective models from low-resolution measurements}
\author[]{A.~Caiazzo$^{\dagger}$, R.~Maier$^*$, D.~Peterseim$^*$}
\address{${}^{\dagger}$ Weierstrass Institute for Applied Analysis and Stochastics (WIAS) Berlin,
	Mohrenstr. 39, 10117 Berlin, Germany}
\email{caiazzo@wias-berlin.de}
\address{${}^{*}$ Department of Mathematics, University of Augsburg, Universit\"atsstr. 14, 86159 Augsburg, Germany}
\email{\{roland.maier,daniel.peterseim\}@math.uni-augsburg.de}
\date{\today}
%
%
\begin{abstract}
We consider the inverse problem of reconstructing an effective model for a~prototypical diffusion process in strongly heterogeneous media based on coarse measurements. 
The approach is motivated by quasi-local numerical effective forward models that are provably reliable beyond periodicity assumptions and scale separation. 
The goal of this work is to show that an identification of the matrix representation related to these effective models is possible. On the one hand, this provides a reasonable surrogate in cases where a direct reconstruction is unfeasible due to a mismatch between the coarse data scale and the microscopic quantities to be reconstructed. On the other hand, the approach allows us to investigate the requirement for a~certain non-locality in the context of numerical homogenization. 
Algorithmic aspects of the inversion procedure and its performance are illustrated in a series of numerical experiments.
\end{abstract}
%
%
\maketitle

\setcounter{tocdepth}{2}
\indent
{\tiny {\bf AMS subject classifications.}  
{\bf 65N30}, 
{\bf 74Q05}, 
{\bf 65N21}} 

%
\section{Introduction}\label{sec:intro}

This paper focuses on the computational solution of multiscale inverse problems, i.e.,  
where the quantities to be sought are related to an unknown microscopic mathematical model, while measurement data are available only on a much coarser scale. 
Due to the scale mismatch between given data and unknowns, the direct recovery of microscopic quantities, e.g., in the form of a coefficient of a partial differential equation (PDE), is not only expensive but may also lead to unsatisfactory results. Therefore, the aim of this paper is to introduce a computational framework -- inspired by numerical homogenization methods -- to reconstruct an effective model, i.e., an alternative quantity on the coarse scale of the available data.

Effective models are the key to bridge the discrepancy between a microscopic coefficient and a coarse scale of interest in the forward setting. They provide models that compute reliable approximations of the solution of a PDE even in the presence of microscopic quantities. If structural assumptions such as (local) periodicity or scale separation hold, \emph{classical} homogenization methods (see, e.g.,  \cite{MR1455261,matache2002two,ee2003heterogenous,ee2005heterogeneous}) based on analytical homogenization theory can be used. These methods are \emph{local} in the sense that the communication among the degrees of freedom is only between neighbors. In a more general setting where these structural assumptions do not hold a~priori, \emph{numerical} homogenization methods (see, e.g., \cite{MalP14,HenP13,GGS12,BOZ13,EGH13,O17,Mai20}) provably provide an alternative. These methods are based on a coarse mesh with a characteristic mesh parameter and compute special problem-adapted basis functions with optimal approximation properties. Compared to the locality of classical homogenization methods, numerical homogenization methods typically involve a slight deviation from local communication between the degrees of freedom which, in turn, leads to somewhat increased sparsity patterns of the corresponding system matrices. Since this non-locality can be controlled, we refer to these methods as \emph{quasi-local}. 

This paper follows the pragmatic approach of reconstructing quasi-local effective models (i.e., their representation in term of quasi-local system matrices) that describe the effective behavior of a medium with microstructures based on coarse (i.e., low-resolution) measurement.  
On the one hand, this provides a surrogate to a direct reconstruction in the above-mentioned multiscale context. On the other hand, the reconstruction of an effective model allows us to investigate the requirement for a certain quasi-locality in the context of numerical homogenization approaches. 

The goal of this work is to promote this idea along with algorithmic aspects and preparatory numerical experiments. 
To demonstrate the feasibility and the potential of the approach, we investigate a stationary linear elliptic multiscale diffusion problem.
Moreover, we consider a worst-case scenario without any structural a priori knowledge on the underlying diffusion coefficient and do not assume that the heterogeneous coefficient can be parameterized by a few unknown parameters that could more easily be identified. 
The main novelty of our approach is that it aims to recover information about the microscopic scale in the sense of reproducing the effective behavior of corresponding solutions, instead of aiming at identifying the actual microscopic coefficient.
In the multiscale setting, in which measurement are given on a coarse scale and without a priori assumption on the structure of the microscopic coefficient, one cannot hope to reasonably identify the actual coefficient and our approach presents an alternative strategy. Our aim is not to compete with classical inverse strategies but to provide a first step towards reasonable surrogates in the case where a direct recovery of the coefficient is either too expensive or not reliable due to multiscale aspects of the problem. 

The remaining parts of the paper are organized as follows. 
We start with introducing the microscopic forward problem and motivating the reconstruction of an effective model, represented by an effective system matrix 
(Section~\ref{sec:forward}). 
This strategy is inspired by numerical homogenization strategies which provably provide reliable effective models for the given forward model. The rigorous motivation is presented in Section~\ref{sec:numhom} and is mainly to emphasize that an effective model indeed exists in the setting of very general coefficients. 
Based on the considerations for the forward problem, we then tackle the reconstruction of an effective quasi-local model from given measurements. To this end, we prescribe a quasi-local sparsity pattern of the system matrices and rephrase the inverse problem as a non-linear least squares problem for which we apply iterative minimization techniques such as the gradient descent or the Gau\ss-Newton method (Section~\ref{sec:inverse}).
In a series of numerical experiments (Section~\ref{sec:numerics}), we show that quasi-local effective models can indeed be reconstructed. 
In particular, we consider the cases where we are given measurements for all possible (coarse) boundary conditions, and also the setting where solutions are only known for a few boundary conditions. 
The aim of the experiments is to show that allowing the model to deviate from locality improves the inversion process and, thus, justifies the previous discussion. 

\section{Microscopic Forward Problem}\label{sec:forward}

In this section, we present the forward model and identify an effective discrete model, which is characterized by an appropriate system matrix. 

\subsection{Problem setting}

We consider the prototypical second-order linear elliptic diffusion problem
\begin{equation}\label{eq:PDEmodel}
\begin{aligned}
-\ddiv A \nabla u &= f &&\text{ in } \Omega,\\
 u &= u^0 &&\text{ on } \partial \Omega,
\end{aligned}
\end{equation}
where $\Omega \subset \R^d$, $d\in\{1,2,3\}$ is a polyhedral domain and the diffusion coefficient $A$ encodes the microstructure of the medium. 
We do not make any structural assumptions on the coefficient such as periodicity or scale separation. 
Admissible coefficients are elements of the following set,
\begin{equation*}
\AC := \left\{\begin{aligned} 
&A \in L^\infty(\Omega;\R_\mathrm{sym}^{d\times d})\,:\,\exists\, 0 < \alpha \leq \beta < \infty\,:\\&\forall \xi \in \R^d,\, \text{a.a. } x \in \Omega\,:\, \alpha |\xi|^2 \leq A(x)\xi \cdot \xi \leq \beta |\xi|^2\end{aligned}\right\},
\end{equation*}
which only requires minimal assumptions. 

Since solutions to problem \eqref{eq:PDEmodel} do not necessarily exist in the classical sense, we are interested in the weak solution of \eqref{eq:PDEmodel} in the Sobolev space $\V:=H^1(\Omega)$ which is characterized by the following variational formulation. 
Given $A\in\AC$, $u^0 \in\X:=H^{1/2}(\partial \Omega)$, and $f \in L^2(\Omega)$, we seek $u \in \V$ such that
\begin{equation}\label{eq:weakformH1}
\begin{aligned}
a(u,v) & = (f,v) &&\text{ for all } v \in \Vo:=H^1_0(\Omega),\\
\tr\, u & = u^0 &&\text{ on } \partial \Omega, 
\end{aligned}
\end{equation}
where $\tr\colon\V \to \X$ is the trace operator, $(w,v) := (w,v)_{L^2(\Omega)}$ denotes the $L^2$ inner product, and $a(w,v) := \int_{\Omega} A \nabla w \cdot \nabla v \dx$. 
Note that instead of \eqref{eq:PDEmodel}, we could as well consider a general second-order linear PDE in divergence form with additional lower-order terms. Such a generalization is straight-forward but is omitted for simplicity.  

\subsection{Coarse discretization}\label{ss:coarsedisc}
In practice, it is favorable to rewrite problem \eqref{eq:weakformH1} as a problem with homogeneous Dirichlet boundary conditions in $\Vo$.
Let $\Eb\colon\X\to\V$ be a linear extension operator, which also defines the restriction operator $\Ri\colon\V\to\Vo$ by $\Ri:=1-\Eb\,\tr$. 
Then, we can decompose $u = \Ri u + (1-\Ri)u = \Ri u + \Eb u^0$ and problem \eqref{eq:weakformH1} reduces to finding $\Ri u \in \Vo$ such that
\begin{equation}\label{eq:weakformH10}
a(\Ri u,v)  = (f,v) - a(\Eb u^0,v) 
\end{equation}
for all $v \in \Vo$.

Let us now introduce a \textit{coarse} target scale $H$ (e.g., the resolution of the data available for the inverse problem). 
We adopt the notation from numerical homogenization where a capital $H$ is used to indicate that the scale is indeed a coarse one.
In typical applications, $H$ will be much larger than the \textit{microscopic} scale, i.e., the scale on which the diffusion coefficient varies.

In order to discretize \eqref{eq:weakformH10}, let $\tri_H$ be a mesh of orthotopes with characteristic mesh size $H$ and denote with $Q^1(\tri_H)$
the corresponding space of piecewise bilinear functions. 
Further, we define the discrete spaces $\VH:=Q^1(\tri_H) \cap \V$, $\VHo := \VH \cap \Vo$, and $\XH := \tr\,\VH$ with dimensions $m = \mathrm{dim}\,\VH$, $m^0 = \mathrm{dim}\,\VHo$, and $n = \mathrm{dim}\,\XH$, respectively. The choice of these finite element spaces is not unique and other standard finite element spaces could be used. 

In the context of inverse problems, it is reasonable to consider that $u^0$ is defined as the first order finite element approximation of coarse experimental boundary data which approximate the real data up to order $H$ in the $H^{1/2}$ norm. That is, in the following we will assume that $u^0 \in \XH$. 
Further, we assume to have a discretized extension operator $\EbH\colon\XH\to\VH$ that fulfills $\Eb\vert_{\XH} = \EbH$ and a corresponding restriction operator $\RiH\colon\VH\to\VHo$ defined by $\RiH := 1 - \EbH\,\tr$.

Based on the above spaces, we introduce an injective linear operator 
\begin{equation}\label{eq:opG}
\calG\colon \VH \to \V \quad\text{ with }\quad \calG \VHo \subseteq \Vo,
\end{equation} 
which leads to the following discretization of \eqref{eq:weakformH10}: find $u_H \in \VH$, $u_H = \RiH u_H + \EbH u^0$, such that $\Ri u_H \in \VHo$ solves
\begin{equation}\label{eq:weakformH10disc}
a(\calG\RiH u_H,\calG v_H)  = (f,v_H) - a(\calG \Eb u^0,\calG v_H) 
\end{equation}
for all $v_H \in \VHo$. Further, we define the solution operator
\begin{equation}\label{eq:operL}
\begin{aligned}
\DtoN_A\colon \XH \times L^2(\Omega) & \to && V,\\
(u^{0},f) &\mapsto&& u, \text{ where } u \text{ solves } \eqref{eq:weakformH1}
\end{aligned} 
\end{equation}
and its discretized version 
\begin{equation}\label{eq:operLell}
\begin{aligned}
\DtoN_{A,\ell}^\calG\colon \XH \times L^2(\Omega) & \to && \VH,\\
(u^{0},f) &\mapsto&& u_H, \text{ where } u_H \text{ solves } \eqref{eq:weakformH10disc}.
\end{aligned} 
\end{equation}
The operator $\DtoN_A$ (and similarly also $\DtoN_{A,\ell}^\calG$) can be written as
\begin{equation}\label{eq:opDec}
\DtoN_A(u^0,f) = \DtoN_A(u^0,0) + \DtoN_A(0,f) 
\end{equation}
with the linear operators $\DtoN_A(\cdot,0)\colon\XH\to V$ and $\DtoN_A(0,\cdot)\colon L^2(\Omega)\to V$. 
For simplicity, we assume in the following that $f$ is a fixed function and only the boundary conditions may change. 
The generalization to the case where $f$ is variable as well is conceptually straightforward but slightly more involved. 
The decomposition \eqref{eq:opDec} motivates the distance function between operators defined by
\begin{equation}\label{eq:distf}
\dist_f(\mathfrak{A},\mathfrak{B}) := \left( \|\mathfrak{A}(\cdot,0)-\mathfrak{B}(\cdot,0)\|^2_{\mathcal{L}(\XH;L^2(\Omega))} + \|\mathfrak{A}(0,f)-\mathfrak{B}(0,f)\|^2_{L^2(\Omega)}\right)^{1/2}
\end{equation}
for all $\mathfrak{A,\,B}\colon \XH\times L^2(\Omega)\to V$. Note, however, that also other distance functions could be used.

\subsection{Characterization in terms of a stiffness matrix}\label{ss:stiffmatrix}
As a next step, we discuss an alternative representation of the operator $\DtoN_{A,\ell}^\calG$ using the stiffness matrix corresponding to the discrete formulation \eqref{eq:weakformH10disc}. 
Given a coefficient $A\in\AC$ and a mapping $\calG\colon\VH\to\V$ as above, the stiffness matrix $S_H=S_H(A,\calG)$ is defined by
\begin{equation}\label{eq:S_H} 
S_H[i,j] := a(\calG\Lambda_{z_j}, \calG\Lambda_{z_i}),\quad i,\, j \in \{1,\dots,m\},
\end{equation}
where $i \mapsto z_i$ is a fixed ordering of the $m$ nodes in $\tri_H$ and $\Lambda_z$ denotes the classical finite element hat function associated with the node $z \in \tri_H$.  

Therefore, we may define the operator
\begin{equation}\label{eq:opLSH}
\begin{aligned}
\DtoN_{S_H}\colon\XH \times L^2(\Omega) & \to && V_H,\\
(u^{0},f) &\mapsto&& u_H, \text{ where } u_H \text{ solves}\\
&&&\hspace{22pt}\left\{\begin{aligned} S_{H,0} \RiH u_H &= R_H M_H f_H - \RiH S_H \EbH u^0,\\
u_H&= u^{0} \text{ on } \partial \Omega,\end{aligned}\right.
\end{aligned}
\end{equation}
with the classical finite element mass matrix $M_H$, the restriction $S_{H,0} =\RiH S_H \RiH^T$ of $S_H$ to the inner nodes of $\tri_H$, and $f_H := \Pi_H f$ the $L^2$ projection of $f$ onto $\VH$. 
For better readability, we use the notation $v_H$ (or $B_H$) for both the vector $v_H \in \R^{m}$ (or the matrix $B_H \in \R^{m^0\times m}$) and the corresponding function $v_H \in \VH$ (or the mapping $B_H\colon\VH\to\VHo$). 
	
The following theorem is the basis of the inverse strategy discussed in the next section. It is a direct consequence of Lemma~\ref{l:LH=Lell} and Corollary~\ref{cor:errForward}, which are proven in Section~\ref{sec:numhom}. It states the existence of an appropriate coarse model, characterized by an operator $\calG$, that is able to capture the effective behavior of the original model \eqref{eq:operL}. 

\begin{theorem}[Existence of an effective model]\label{t:errEff}
There exists an operator $\calG$ as in \eqref{eq:opG} and a corresponding stiffness matrix $S_H(A,\calG)$ such that 
\begin{equation}
\dist_f\big(\DtoN_A,\DtoN_{S_H(A,\calG)}\big) \lesssim H.
\end{equation}
Moreover, there exists a choice of $\calG$ and $\ell \sim |\log H|$ such that $S_H(A,\calG) \in \MS(\ell,\tri_H)$ with
\begin{equation}\label{eq:defMatrices}
\MS(\ell,\tri_H) := \left\{
S_H \in \R^{m\times m}_\mathrm{sym}\,:\, \forall\, 1 \leq i\leq j \leq m\,: \,z_i \notin \mathrm{N}^\ell({z_j})  \Rightarrow  S_H[i,j] = 0
\right\}
\end{equation}
defined as the set of matrices that may have a non-zero entry at position $[i,j]$ only if the corresponding nodes $z_i$ and $z_j$ belong to the $\ell$-neighborhood of each other. The $\ell$-neighborhood $\mathrm{N}^\ell$ is defined by
\begin{equation}\label{eq:Nb}
\mathrm{N}^\ell(\omega) := \mathrm{N}(\mathrm{N}^{\ell-1}(\omega)),\,\ell \geq 1,\qquad \mathrm{N}^0(\omega) := \bigcup \bigl\{T \in \tri_H\,:\, \omega\cap\overline T\subseteq\overline T \bigl\}
\end{equation}
for $\omega \subseteq \Omega$. 

We call the operator $\DtoN_{S_H(A,\calG)}$ the \emph{effective model} and $S_H(A,\calG)$ the \emph{effective stiffness matrix}. 
\end{theorem}
Theorem~\ref{t:errEff} justifies the inverse procedure that is presented in the following section. In particular, we may interpret coarse measurements of a solution operator $\DtoN_A$ as approximations obtained by an effective model due to the fact that there exists an appropriate effective model which is reasonably close.
\begin{figure*}
	\centering
	\captionsetup[subfigure]{labelformat=empty}
	\begin{subfigure}[b]{0.3\textwidth}
		\centering
		\includegraphics[width=\textwidth]{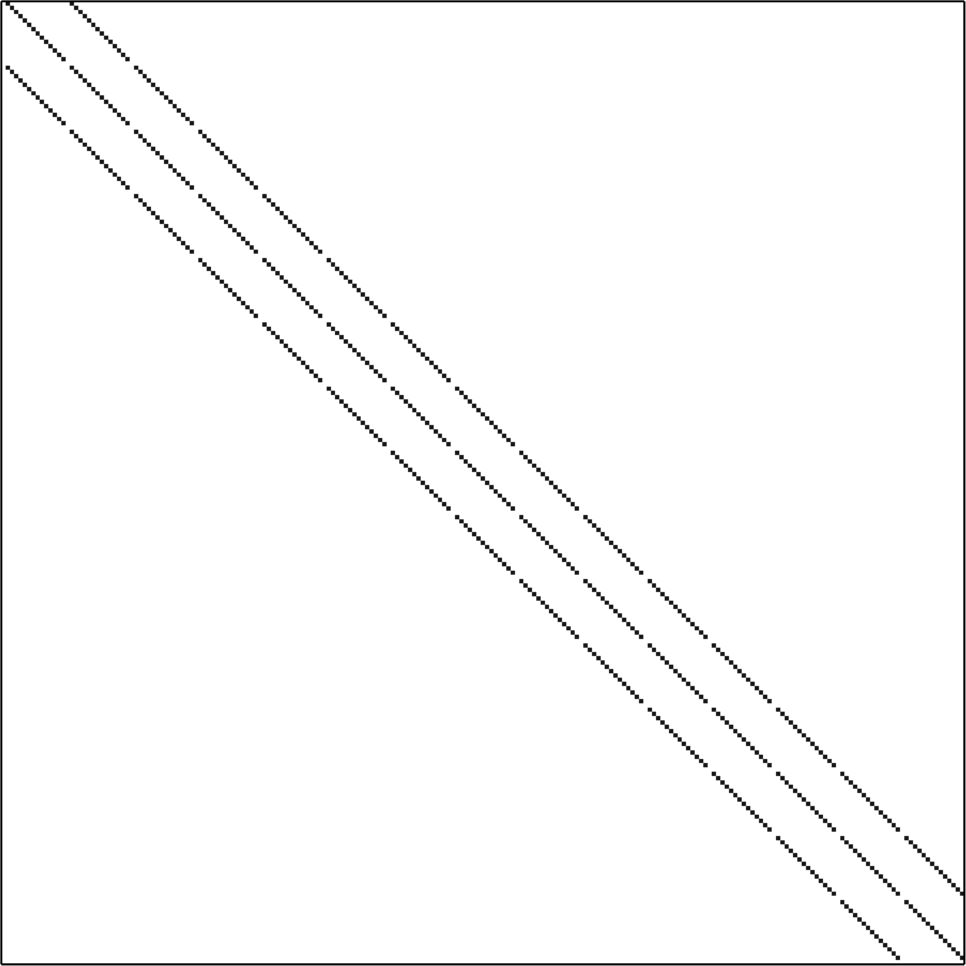}
		\caption
		{\small $\ell = 0$}
		\label{fig:pattern0}    
	\end{subfigure}
	\qquad
	\begin{subfigure}[b]{0.3\textwidth}  
		\centering 
		\includegraphics[width=\textwidth]{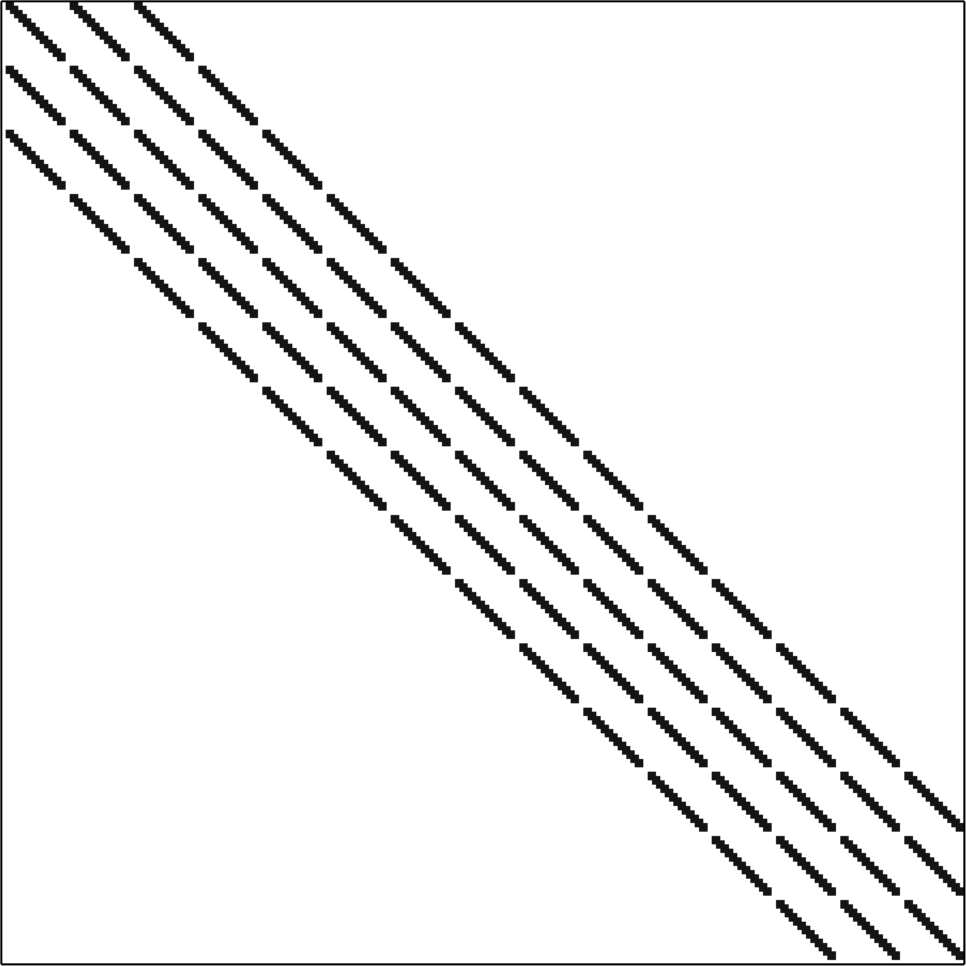}
		\caption[]%
		{\small $\ell=1$}   
		\label{fig:pattern1}
	\end{subfigure}
	\vskip\baselineskip
	\begin{subfigure}[b]{0.3\textwidth}   
		\centering 
		\includegraphics[width=\textwidth]{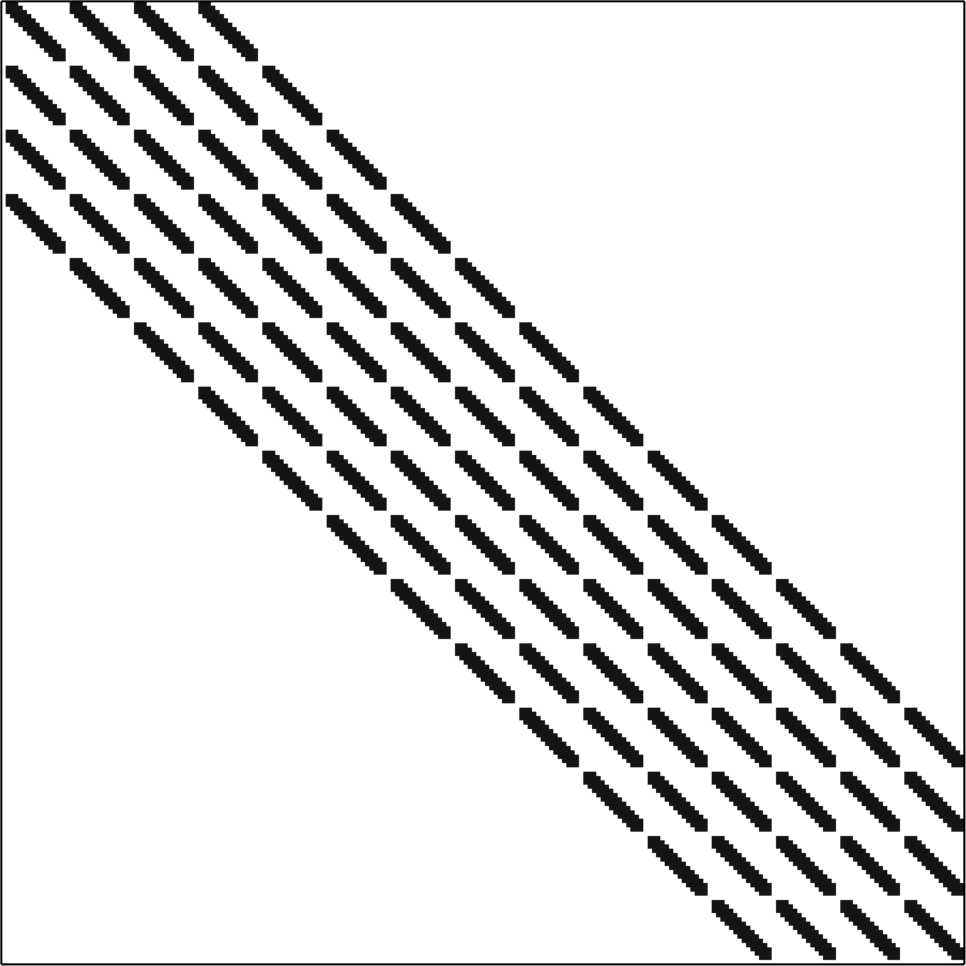}
		\caption[]%
		{\small $\ell=2$}    
		\label{fig:pattern2}
	\end{subfigure}
	\qquad
	\begin{subfigure}[b]{0.3\textwidth}   
		\centering 
		\includegraphics[width=\textwidth]{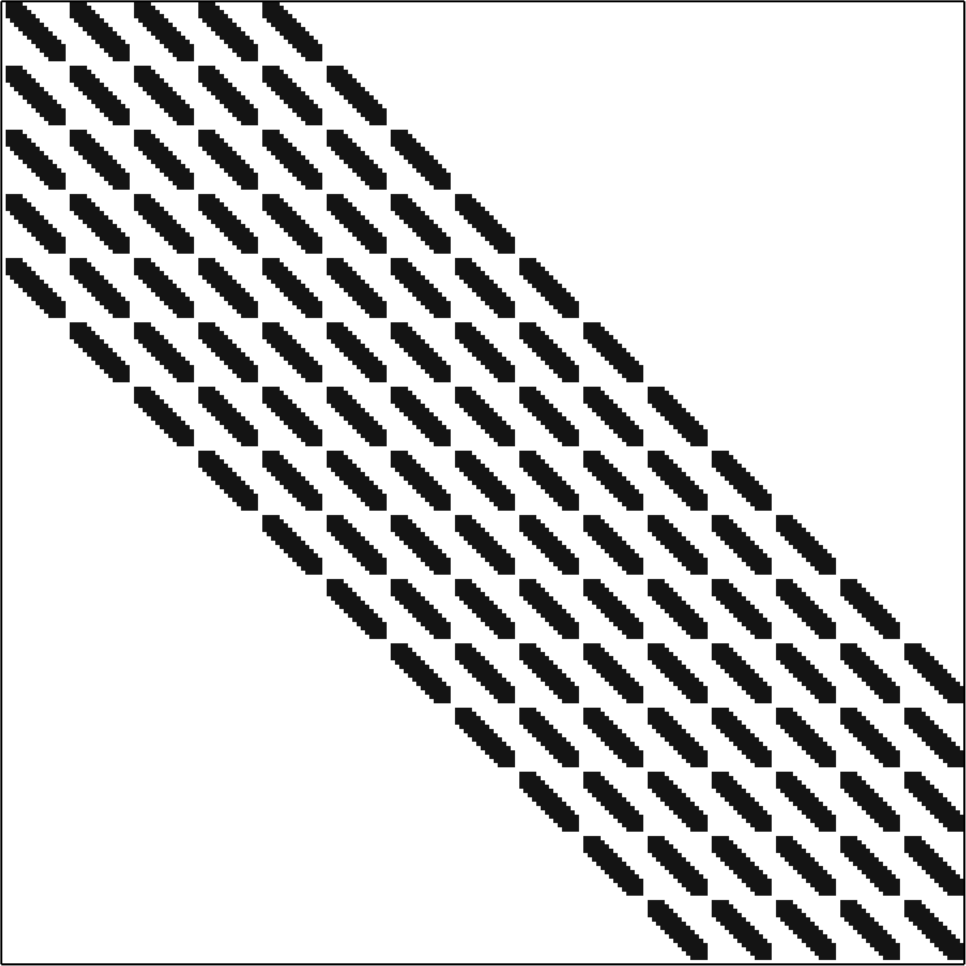}
		\caption[]%
		{\small $\ell=3$}    
		\label{fig:pattern3}
	\end{subfigure}
	\caption[]
	{\small Sparsity patterns of matrices in $\MS(\ell,\tri_H)$ for different values of $\ell$ on a Cartesian grid ($H = 2^{-4}$) with lexicographic ordering in two dimensions.} 
	\label{fig:sparsity}
\end{figure*}

\subsection{Quasi-locality and connection to numerical homogenization}
The operator $\calG$ defined in \eqref{eq:opG} and the corresponding effective stiffness matrix defined in \eqref{eq:S_H} are strongly related to \emph{numerical homogenization methods}. In contrast to \emph{analytical homogenization}, these approaches provably work beyond structural assumptions such as (local) periodicity or a clear separation of scales which cannot be guaranteed for general microstructures. 

One such method is the Localized Orthogonal Decomposition (LOD) approach which provides effective models that provably cope with arbitrary rough coefficients in a large class of model problems including diffusion problems \cite{MalP14,HenP13,HMP12}, elasticity \cite{HenPer16,ACMPP18} and wave propagation \cite{GalPet15,AH17,P17,Ver17,GHV17,MP18}, without requiring periodicity or scale separation. This method allows us to explicitly characterize an operator $\calG$ to prove Theorem~\ref{t:errEff} in Section~\ref{sec:numhom}.
For linear elliptic problems, there are various other numerical homogenization approaches such as the Generalized Finite Element Methods (GFEM) \cite{BL11}, AL bases \cite{GGS12}, Rough Polyharmonic Splines (RPS) \cite{BOZ13}, the Generalized Multiscale Finite Element Method (GMsFEM) \cite{EGH13}, Gamblets \cite{O17}, CEM-GMsFEM \cite{CEL18}, the higher-order multiscale approach described in \cite[Ch.~3]{Mai20}, and their variants with similar properties as LOD. 
All these methods compute special problem-adapted basis functions with optimal approximation properties based on underlying Galerkin methods.
To achieve optimal accuracy, a moderate price in terms of the computational complexity has to be paid compared to a standard finite element method (fixed order) on the same mesh in order to account for microscopic information. 
The computational overhead is either characterized by an increase in the number of degrees of freedom per mesh entity (GFEM, GMsFEM), e.g., elements or nodes, or in enlarging the support of the discrete basis functions (LOD, RPS, Gamblets, AL bases).  
In both cases, the result is a slightly denser sparsity pattern of the corresponding system matrices which is due an increased communication between the degrees of freedom. 
This quasi-locality distinguishes the above numerical homogenization methods from classical numerical multiscale methods based on homogenization theory such as the Multiscale Finite Element Method (MsFEM) \cite{MR1455261}, the Two-Scale Finite Element Method \cite{matache2002two}, or the Heterogeneous Multiscale Method (HMM) \cite{ee2003heterogenous,ee2005heterogeneous} that share the communication pattern of standard finite element methods. 

Quasi-locality also showed to be viable in connection with the pollution effect in high-frequency time-harmonic wave propagation \cite{BS97} which cannot be avoided unless the mesh size is coupled to, e.g., the polynomial degree \cite{MS10,MS11,MPS13} or the support of the basis functions \cite{P17} in a logarithmic way. 
Promising results using non-local models have also been achieved in the field of peridynamics \cite{Sil00,Lip14,Du16} or in isogeometric analysis \cite{IGA05,IGA09}. 
Moreover, quasi-locality also includes the increased communication of higher-order finite element approaches. 

\section{Inverse Problem: Reconstruction of an Effective Model}\label{sec:inverse}

Based on the considerations in the forward setting, we are now able to formulate the inverse problem which reconstructs an effective model characterized by a stiffness matrix with a certain sparsity behavior.

\subsection{Problem setting}\label{ssec:setting}

Let us assume that the diffusion coefficient $A$ is unknown and that structural assumptions
such as periodicity, quasi-periodicity, and given parameterization by few degrees of freedom are not satisfied a priori.
In an ideal setting, information about solutions to problem \eqref{eq:weakformH1} in the form of a solution operator
\begin{equation}\label{eq:solop}
\tilde\DtoN:=\DtoN_A(\cdot,f)\colon \X \to \V
\end{equation}
would be given. In practical applications, however, boundary data and information about the corresponding solutions are only available on some (coarse) scale, possibly much larger than the (micro) scale on which the diffusion coefficient and the corresponding solutions vary. 
A classical formulation of the inverse problem, for a fixed right-hand side $f$, consists in recovering $A$ in \eqref{eq:weakformH1} given the mapping
\begin{equation}\label{eq:solopeff}
\tilde\DtoN^\mathrm{eff}:=\DtoN^\mathrm{eff}_A(\cdot,f)\colon \XH \to \VH
\end{equation}
which comprises coarse measurements of solutions to \eqref{eq:weakformH1}.

However, since the unknown coefficient $A$ includes microscopic features, a direct approach of recovering $A$ by simulations on the microscopic scale is computationally unfeasible and does not provide reasonable reconstructions due to the mismatch between coarse data and a possibly microscopic coefficient. 
In this section, we present an alternative approach to recover information about the (macroscopic) effective model
taking into account the presence of a micro-scale diffusion coefficient.
Rather than reconstructing the diffusion coefficient itself, we tackle the reconstruction of an effective stiffness matrix that is able to reproduce the given data related to solutions to \eqref{eq:weakformH1}.
Therefore, the alternative formulation of the inverse problem reads:
\begin{equation}\label{eq:inv} 
\text{given }\tilde\DtoN^\mathrm{eff}\colon \XH \to \VH,\text{ find the corresponding stiffness matrix }\tilde S_H.
\end{equation} 
This alternative inverse problem is built upon the assumption that coarse given measurements on the scale $H$ may be represented by an effective mapping (up to order $H$) based on Theorem~\ref{t:errEff}.

The effective model corresponding to the reconstructed matrix provides a reliable surrogate that contains effective properties and that may be used for simulations on the coarse data scale. 
Moreover, knowledge on numerical homogenization methods can, in a second step, provide further information about, e.g., geometric features.

\subsection{Medical background}\label{ssec:medical}
The considered inverse problems are of internal type, i.e., we assume that measurement are available on the bulk domain.
This setting is motivated by the use of medical imaging protocols based on Magnetic Resonance Imaging (MRI), which
play a key role in modern diagnostics.
Using strong magnetic fields, magnetic field gradients, and radio waves, MRI allows clinicians to obtain, in vivo and non-invasively (and without exposing the body to radiation), both geometrical features of the body, e.g., shapes of the organs,
and functional properties, e.g., motion or diffusion processes.
As an example, Magnetic Resonance Elastography (MRE) is an imaging technique which is sensitive to mechanical parameters of the tissue.
During an MRE examination, the tissue undergoes an external mechanical excitation, imposed at given frequencies 
by so-called \textit{actuators}, attached to the surface of human tissues. 
In parallel, using phase-contrast MRI, i.e., postprocessing
the phase of the MRI signal, it is then possible to measure the \textit{internal} displacement of the tissue and hence to 
recover how shear and compression waves propagate into the body \cite{muthupillai-96, sack-08, sack-book-17}.
In practice, MRE allows to obtain average displacement fields on each element of a three-dimensional Cartesian mesh (a \textit{voxel}), whose 
resolution is typically of the order of millimeters, and it is practically limited by the examination time and by the properties of the MRI scanner.

Combining these data with a mechanical tissue model, it is then
possible to recover information about the elastic behavior of the tissue (a so-called \textit{elastogram}).
Clinical application of MRE are based on reconstructing tissue properties only on an \textit{effective} scale, i.e., describing the living tissue as a (visco-)elastic material with mechanical parameters varying only on the coarse data scale. The procedure is currently used to diagnose and monitor tissue diseases such as cancer and fibrosis, that
are characterized by different tissue stiffnesses. 
However, 
parameters describing the microscopic scales -- such as tissue porosity and vascular structures -- might play an important role in several applications, especially
in the context of poroelastic tissues \cite{hirsch-etal-2013-compmre,hirsch-etal-2014-liver}.
In those cases, in order to characterize the microstructural properties from coarse data, 
mathematical models defined on the effective scale -- such as the one proposed in Section \ref{ssec:setting} -- might provide a valuable alternative to efficiently take smaller scales into account.

\subsection{The minimization problem}

We formulate the inverse problem \eqref{eq:inv} as a minimization problem in the set $\MS(\ell,\tri_H)$, i.e., 
\begin{equation}\label{eq:invprob}
\text{find } \tilde S_H^* = \arg\min_{S_H \in \MS(\ell,\tri_H)} \mathcal{\tilde J}_H(S_H),
\end{equation}
where
\begin{equation}\label{functJ}
\mathcal{\tilde J}_H(S_H) = \frac{1}{2}\Big(\dist_f(\tilde\DtoN^\mathrm{eff},\DtoN_{S_H})\Big)^2.
\end{equation}
Based on Theorem~\ref{t:errEff}, we interpret the operator $\DtoN_{S_H}(\cdot,f)\colon \XH \to \VH$ as an effective model which can be represented by a matrix of size $m \times n$, i.e.,
\begin{equation*}\label{opMatrix}
\DtoN_{S_H}^\mathrm{eff} = \DtoN_{S_H}(\cdot,f) = \left(\id - \RiH^T S_{H,0}^{-1} \RiH S_H\right) \EbH + \RiH^T S_{H,0}^{-1}R_H M_H F_H,
\end{equation*}
with $F_H := [f_H,f_H,\dots,f_H] \in \R^{m \times n}$ and the identity matrix $\id \in \R^{m \times m}$. The matrix $\DtoN_{S_H}^\mathrm{eff}$ comprises full information about the forward problem in the sense that it includes the solutions of \eqref{eq:opLSH} for a complete set of basis functions of $\XH$.
The operator $\tilde\DtoN^\mathrm{eff}$ may also be interpreted as a matrix, so that the distance between the operators can be measured in general matrix norms. 
This is especially useful since a splitting of the form \eqref{eq:opDec} is generally not known for $\tilde\DtoN^\mathrm{eff}$.

Let $\mu := \mathrm{dim}\,\MS(\ell,\tri_H)$. 
Instead of \eqref{eq:invprob}, based on the matrix representation introduced above we consider a minimization problem for the functional $\mathcal{J}_H\colon \R^{\mu}\to \R$ defined by
\begin{equation}\label{functJmatrix}
\mathcal{J}_H(S_H) = \frac{1}{2}\big\| \tilde \DtoN^\mathrm{eff}\big\|_{\R^{{m} \times {n}}}^{-2}\big\| \tilde \DtoN^\mathrm{eff} - \DtoN_{S_H}^\mathrm{eff}\big\|^2_{\R^{{m} \times {n}}}. 
\end{equation}
At this stage, the choice of the norm in $\R^{{m} \times {n}}$ in \eqref{functJmatrix} is arbitrary. 
The results that we will show in Section~\ref{sec:numerics} have been obtained using 
the Frobenius norm, which seems to be a~natural candidate.

\subsection{Iterative minimization}

In order to find a minimizer of \eqref{functJmatrix}, we can now apply standard minimization techniques such as the \emph{Newton method} or the \emph{gradient descent method}. 
Here, we adopt a \emph{Gau\ss-Newton method} which, in our numerical computations, showed faster convergence in terms of number of iterations. 

In order to compute the descent direction, the most important step concerns the computation of the gradient of $\mathcal{J}_H$ with respect to the \emph{relevant entries} $\{s_i\}_{i=1}^\mu$ of $S_H$ (i.e., the diagonal and the non-zero entries above the diagonal, due to symmetry). 
Using the chain rule, we obtain
\begin{equation}\label{gradFunctJ}
\frac{\partial}{\partial s_i} \mathcal{J}_H(S_H) = -\big\| \tilde \DtoN^\mathrm{eff}\big\|_{\R^{{m} \times {n}}}^{-2}\Big(\tilde\DtoN^\mathrm{eff} - \DtoN_{S_H}^\mathrm{eff}\Big) : \frac{\partial \DtoN_{S_H}^\mathrm{eff}}{\partial s_i}. 
\end{equation}
For the Gau\ss-Newton method, only the derivatives of $\DtoN^\mathrm{eff}_{S_H}$ are needed, i.e.,
\begin{equation*}
\begin{aligned}
\frac{\partial \DtoN^\mathrm{eff}_{S_H}}{\partial s_i} &= - \RiH^T \left(\frac{\partial S_{H,0}^{-1}}{\partial s_i}\right) \RiH (S_H \EbH - M_H F_H)
- \RiH^T S_{H,0}^{-1} \RiH \left(\frac{\partial S_H}{\partial s_i}\right) \EbH\\
&= \RiH^T S_{H,0}^{-1}\left(\frac{\partial S_{H,0}}{\partial s_i}\right)S_{H,0}^{-1} \RiH (S_H \EbH - M_H F_H)
- \RiH^T S_{H,0}^{-1} \RiH \left(\frac{\partial S_H}{\partial s_i}\right) \EbH.
\end{aligned}
\end{equation*} 
Here, the double dot product is defined by $M:\tilde M=\mathrm{trace}({M\tilde M^T})$. 
The derivatives $\frac{\partial S_H}{\partial s_i}$ and $\frac{\partial S_{H,0}}{\partial s_i}$ are relatively easy to compute, as they are defined as global matrices that only contain at most two entries equal to $1$. 

For ease of notation, let us interpret $\DtoN_{S_H}^\mathrm{eff}$ and $S_H$ as vectors in $\R^{mn}$ and $\R^{m^2}$, respectively.
The Gau\ss-Newton method to minimize the functional $\mathcal{J}_H$ is then defined by the following steps.
\begin{itemize}
\item Let an initial matrix $S_H^0 \in \MS(\ell,\tri_H)$ be given.
\item For $k=0,1,\hdots$ (until a certain stopping criterion is satisfied), solve
\begin{equation}\label{GNreg1}
H_k p_k= \left[D\DtoN_{S_H^k}^\mathrm{eff}\right]^T \left[\tilde\DtoN^\mathrm{eff}-\DtoN_{S_H^k}^\mathrm{eff}\right] 
\end{equation}
where $D$ denotes the derivative with respect to the relevant entries of $S_H$ and
\begin{equation*}
H_k = \left[D\DtoN_{S_H^k}^\mathrm{eff}\right]^T\left[D\DtoN_{S_H^k}^\mathrm{eff}\right].
\end{equation*} 
\item Set $P_k\in\MS(\ell,\tri_H)$ as the matrix whose relevant entries are given by $p_k$ and define
\begin{equation}
S_H^{k+1} = S_H^{k} + \delta_k P_k
\end{equation}
with appropriately chosen \emph{step size} $\delta_k$, for example using \emph{backtracking line search} based on the \emph{Armijo-Goldstein condition}. 
\end{itemize}
Due to the ill-posedness of the inverse problem, the matrix $H_k$ might be singular. 
A possible approach to overcome this issue consists in replacing \eqref{GNreg1} by 
\begin{equation}\label{GNreg2}
\left(H_k + \eta\id\right) p_k= \left[D\DtoN_{S_H^k}^\mathrm{eff}\right]^T \left[\tilde\DtoN^\mathrm{eff}-\DtoN_{S_H^k}^\mathrm{eff}\right]
\end{equation}
with a given parameter $\eta > 0$.  

Another possible strategy consists in adding a regularization term to the functional to be minimized, i.e., in replacing \eqref{functJmatrix} by
\begin{equation}\label{J+Sreg}
\mathcal{J}_H(S_H) = \frac{1}{2}\big\| \tilde \DtoN^\mathrm{eff}\big\|_{\R^{{m} \times {n}}}^{-2}\big\| \tilde \DtoN^\mathrm{eff} - \DtoN_{S_H}^\mathrm{eff}\big\|^2_{\R^{{m} \times {n}}}
+\frac{\gamma}{2} 
\big\| S_\mathrm{reg} - S_H\big\|^2_{R^{m \times m}}
\end{equation}
where $\gamma > 0$ is a given regularization parameter and $S_\mathrm{reg}$ is a regularization (or stabilization) matrix.
Additionally, the computations of the gradient in \eqref{gradFunctJ} need to be adapted accordingly.
In the presence of multiple minimizers, this regularization enforces the solution to be close (depending on the parameter $\gamma$) to the matrix $S_\mathrm{reg}$. 
For example, if the aim of the inverse problem is to find defects in an otherwise homogeneous medium, a suitable choice for $S_\mathrm{reg}$ could be a standard finite element stiffness matrix for a constant diffusion coefficient.
In our practical computations, the regularization approach described in \eqref{GNreg2} is better suited since an appropriate regularization matrix $S_\mathrm{reg}$ is generally not known. 

We emphasize that the presented inversion process does not need to resolve any microscopic scales in order to obtain an effective numerical model. 
The information extracted by this procedure (i.e., the stiffness matrix $\tilde S_H$) may be used to simulate other problems subject to the same (unknown) diffusion coefficient. 
Furthermore, the information gathered can be seen as an intermediate step towards recovering information concerning the original coefficient itself. 
This additional recovery step will be studied in a future work.

\section{Justification of the Proposed Inverse Problem}\label{sec:numhom}

This section is devoted to proving Theorem~\ref{t:errEff} in Section~\ref{sec:forward}. In particular, we explicitly construct an operator $\calG$ as introduced in \eqref{eq:opG}. We emphasize, however, that the explicit construction is only to justify the reconstruction of an effective model. Apart from its communication behavior, none of the specific properties used below are required for the inverse procedure.

\subsection{Effective forward approximation beyond structural assumptions}

In this subsection, we use the multiscale technique known as Localized Orthogonal Decomposition (LOD) \cite{MalP14,HenP13} to obtain a coarse forward model on the scale $H$ which produces a~forward operator that can be used as a replacement for $\DtoN_A$ as given in \eqref{eq:operL}. 

To this end, we discretize \eqref{eq:weakformH10} in a suitable coarse multiscale space.  
Since the standard space $\VH$ is not suitable for the approximation of $u$ if $H$ is larger than the spatial scale of the microstructure, we enrich the coarse model with microscopic information about the problem via corrections of classical finite element functions. 
The construction of these corrections is based on a projective quasi-interpolation operator $I_H\colon \Vo\to \VHo$ with standard approximation and stability properties, i.e., for an element $T \in \tri_H$ with diameter $H_T$, it holds that
\begin{align}\label{eq:IH}
\|H_T^{-1}(v-I_H v)\|_{L^2(T)} + \|\nabla I_H v\|_{L^2(T)} \leq C \|\nabla v\|_{L^2(\mathrm{N}(T))}
\end{align}
for all $v\in \Vo$, where the constant $C$ is independent of $H$, and $\mathrm{N}(T) := \mathrm{N}^1(T)$ is the neighborhood of $T$ as defined in \eqref{eq:Nb}.
Note that for shape-regular meshes the above estimate also holds globally. 
For a particular choice of $I_H$, see \cite{ern2017finite,GP17,KorPY18}.

Based on $I_H$, we define, for any element $T \in \tri_H$ and any function $v_H\in\VHo$, the element correction $\cor_T v_H \in W := \Ker I_H$ by
\begin{equation}\label{eq:cor}
a(\cor_T v_H, w) = \int_T A \nabla v_H \cdot \nabla w =: a_T(v_H, w)
\end{equation}
for all $w \in W$, and the full correction $\cor\colon \VHo \to W$ by 
\begin{equation*}
\cor := \sum_{T \in \tri_H}\cor_T.
\end{equation*}
By construction, it holds that
\begin{equation}\label{eq:ortho}
a((1-\cor)v_H,w) = 0
\end{equation} 
for all $v_H \in \VHo$ and $w\in W$. 
The corrections $\cor_T v_H$ have, in general, global support. 
However, as shown in \cite{HenP13,Pet16} (based on \cite{MalP14}) they decay exponentially fast (see also the one-dimensional sketch in Figure~\ref{fig:1dLODfct}).
\begin{figure*}
	\centering
	\captionsetup[subfigure]{labelformat=empty}
	\begin{subfigure}[b]{0.4\textwidth}
		\centering
		\includegraphics[width=\textwidth]{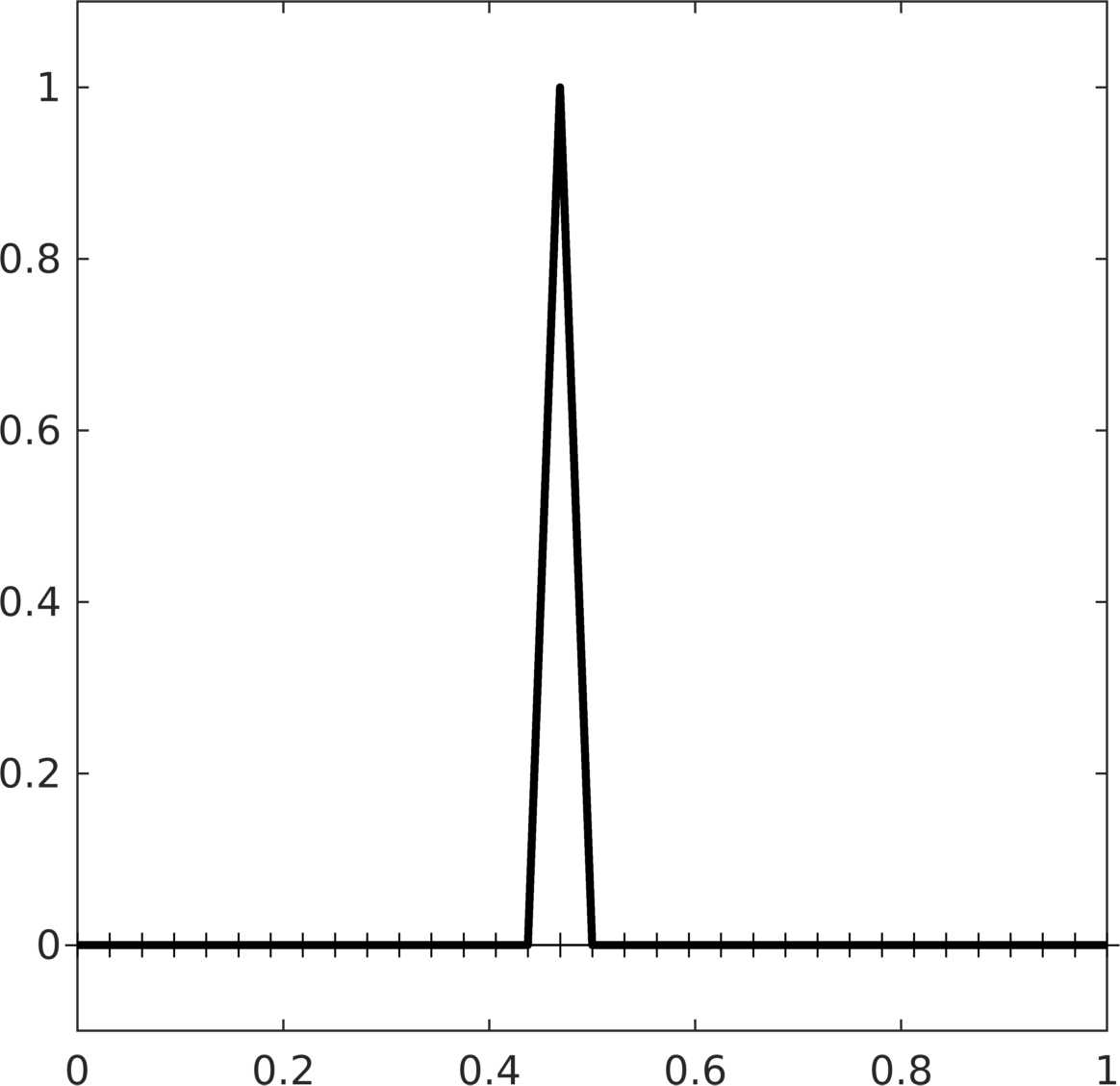}
		\caption
		{\small Hat function $\Lambda$}
		\label{fig:hat}    
	\end{subfigure}
	\quad
	\begin{subfigure}[b]{0.4\textwidth}  
		\centering 
		\includegraphics[width=\textwidth]{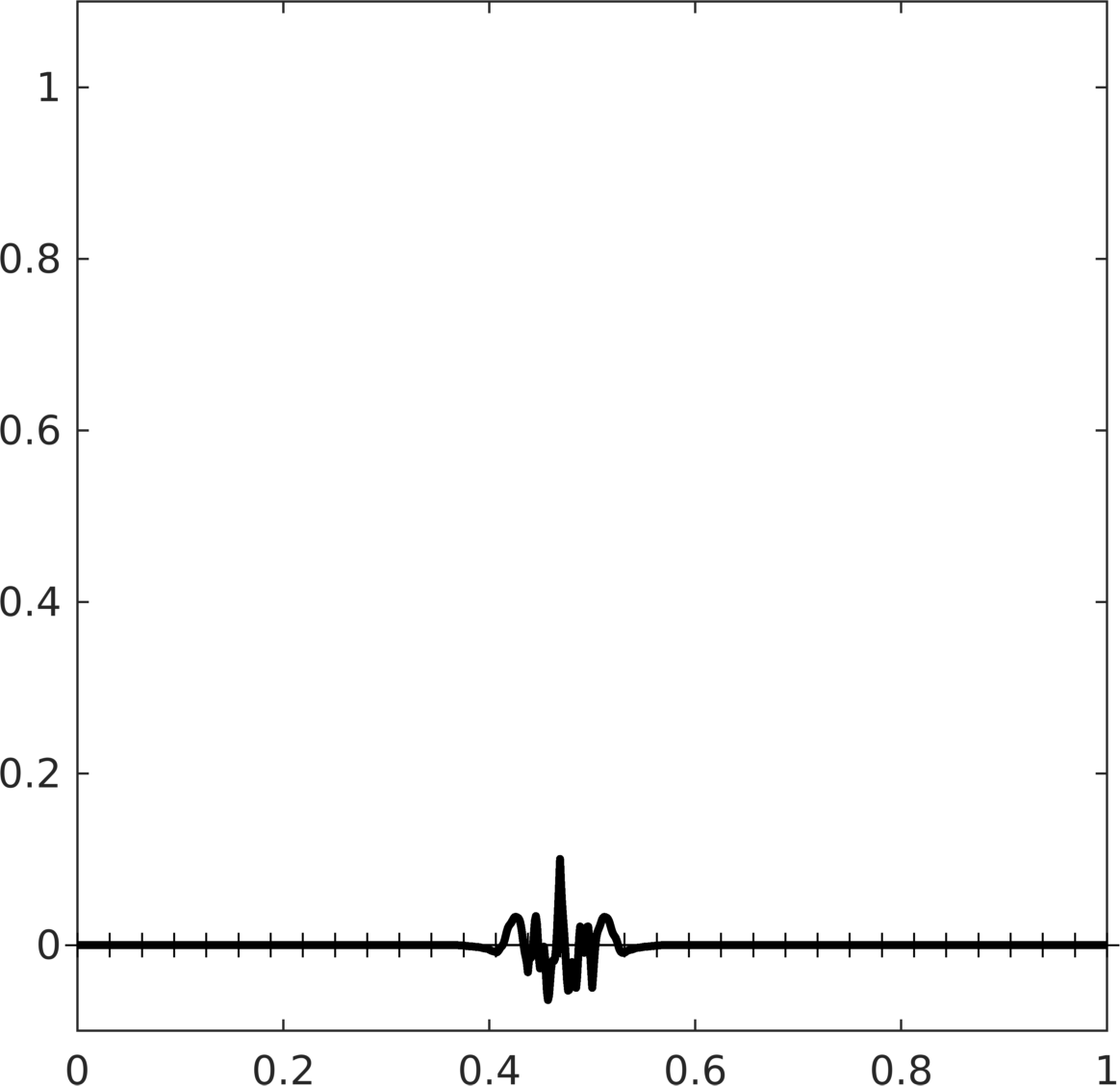}
		\caption[]%
		{\small Correction of hat function $\cor\Lambda$}   
		\label{fig:cor}
	\end{subfigure}
	\vskip\baselineskip
	\begin{subfigure}[b]{0.4\textwidth}   
		\centering 
		\includegraphics[width=\textwidth]{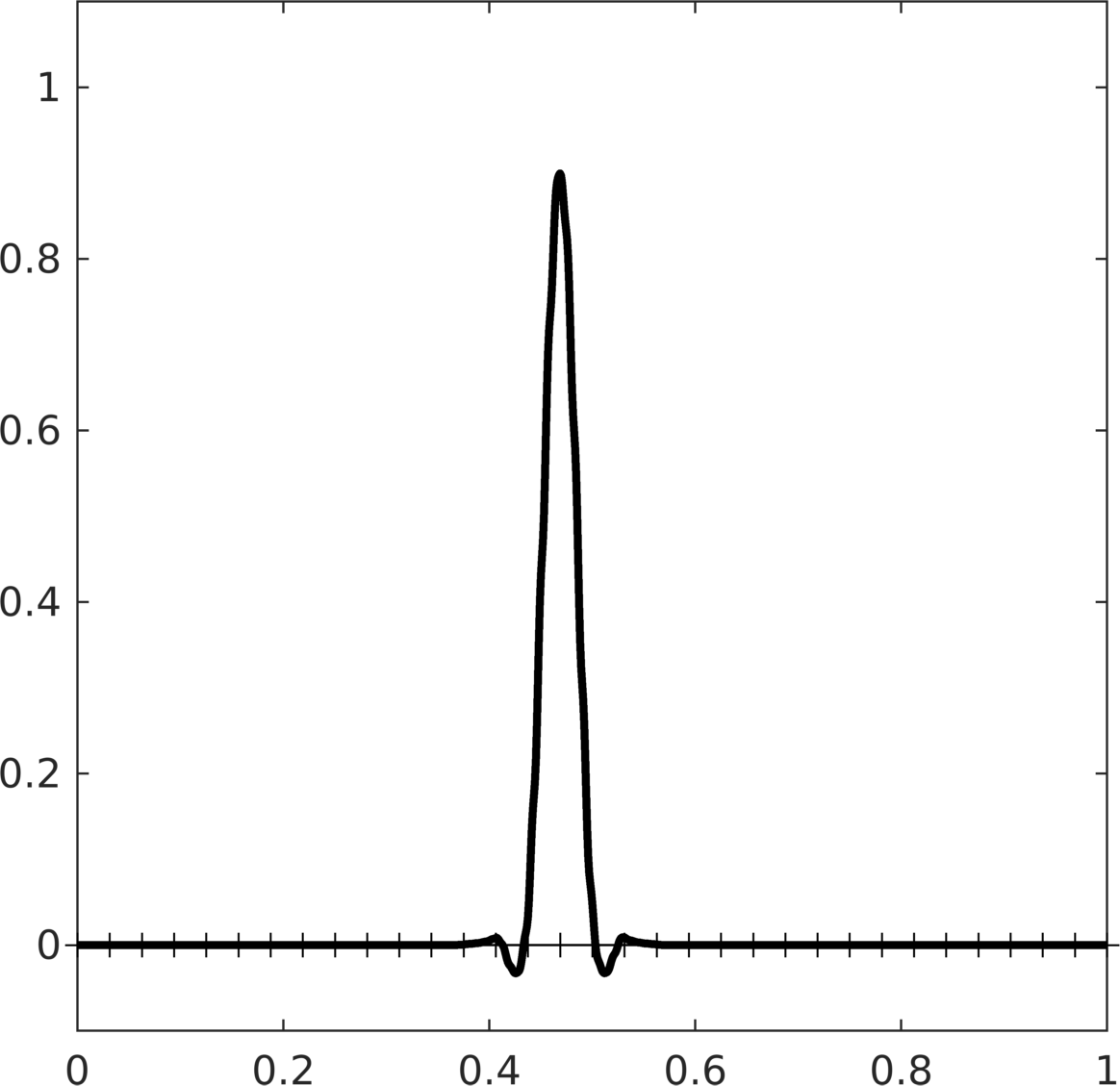}
		\caption[]%
		{\small Corrected hat function $(1-\cor)\Lambda$}    
		\label{fig:hatLOD}
	\end{subfigure}
	\quad
	\begin{subfigure}[b]{0.41\textwidth}   
		\centering 
		\includegraphics[width=\textwidth]{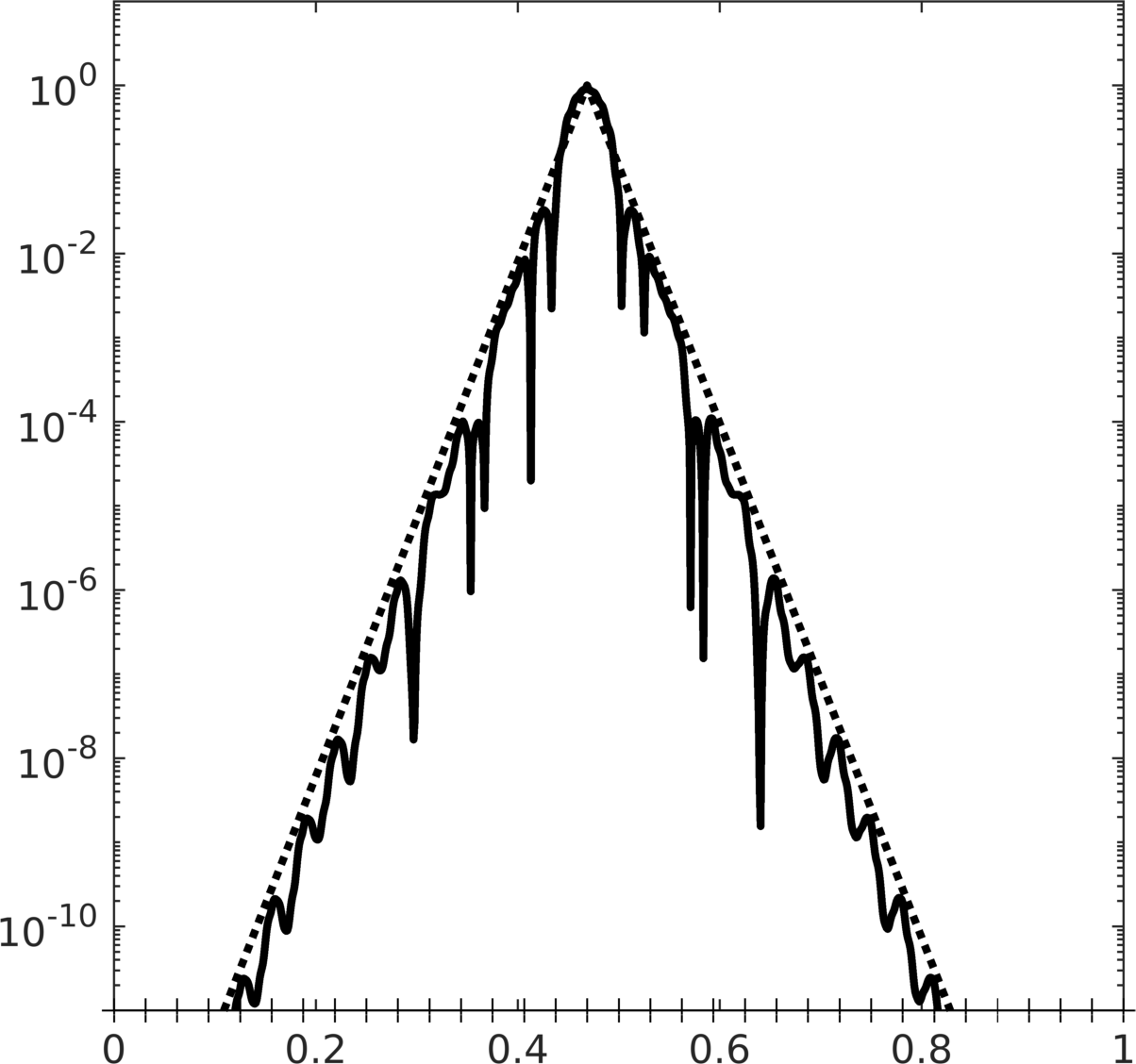}
		\caption[]%
		{\small $|(1-\cor)\Lambda|$ in logarithmic scale}    
		\label{fig:decay}
	\end{subfigure}
	\caption[]
	{\small Illustration of a one-dimensional hat function and its correction for the coefficient $A(x) = (2+\sin(2^8\pi x))^{-1}$.} 
	\label{fig:1dLODfct}
\end{figure*}
Therefore, we use localized element corrections $\cor_{T,\ell}v_H$ which are obtained by solving
\eqref{eq:cor} on local patches $\mathrm{N}^\ell(T)$ as defined in \eqref{eq:Nb}, i.e., 
\begin{equation}\label{eq:corEll}
a(\cor_{T,\ell} v_H, w) = a_T(v_H, w)
\end{equation}
for all $w \in W$ with $w\vert_{\Omega\backslash\mathrm{N}^\ell(T)} = 0$. 
As above, we define the full correction $\cor_\ell\colon \VHo \to W$ by 
\begin{equation*}
\cor_\ell := \sum_{T \in \tri_H}\cor_{T,\ell}.
\end{equation*}
As shown in \cite{HenP13}, we get, for any $v_H \in \VHo$,
\begin{equation}\label{e:diffcor1}
\|\nabla(\cor-\cor_\ell)v_H\|_{L^2(\Omega)} \leq Ce^{-c\ell} \|\nabla v_H\|_{L^2(\Omega)}.
\end{equation}
The constant $c$ only depends on the contrast $\beta/\alpha$, although this dependence seems pessimistic in many cases of practical relevance \cite{PS16,HM17}.
For $v_H \in \VH$, we set $\cor v_H := \cor \Ri v_H$ and $\cor_\ell v_H := \cor_\ell \Ri v_H$.

Given a discretized extension operator $\EbH\colon\XH\to\VH$ 
and the corresponding restriction operator $\RiH\colon\VH\to\VHo$ 
as defined in Section~\ref{ss:coarsedisc}, 
the discretized version of \eqref{eq:weakformH10} reads: 
find $u_H = \RiH u_H + \EbH u_H^0 \in \VH$ such that
\begin{equation}\label{eq:discweakformH10}
a((1-\cor_\ell)\RiH u_H, (1-\cor_\ell)v_H) = (f_H,v_H) - a(\EbH u^0,(1-\cor_\ell)v_H)
\end{equation}
for all $v_H \in \VHo$. 
In \eqref{eq:discweakformH10}, $f_H := \Pi_H f$ is the $L^2$ projection of $f$ onto $\VH$, and $u^0 \in \XH$. 

\subsection{Error estimates}

The following theorem shows that the approximation error of the presented approach scales optimally with $H$ and that it is independent of the variations of the diffusion coefficient.

\begin{theorem}[Error of the forward effective model]\label{t:errForward}
	Let $u \in \V$ be the solution of \eqref{eq:weakformH1} and $u_H \in \VH$ the solution of \eqref{eq:discweakformH10}, for given boundary data $u^0 \in \XH$, a right-hand side $f \in  L^2(\Omega)$, as well as an oversampling parameter $\ell$. 
	
	For $g \in L^2(\Omega)$, let $\hat u(g) \in \V$ denote the solution of \eqref{eq:weakformH1} with right-hand side $g$ and boundary condition $u^0 = 0$, and let us introduce the worst-case best-approximation error
	\begin{equation*}\label{eq:wcba}
	\mathbf{wcba}(A,\tri_H) := \sup_{g \in L^2(\Omega)\backslash \{0\}}\,\inf_{v_H \in \VHo}\frac{\|\Ri \hat u(g) - v_H\|_{L^2(\Omega)}}{\|g\|_{L^2(\Omega)}}.
	\end{equation*}
	It holds 
	\begin{equation*}\label{eq:errorForward}
	\| u - u_H \|_{L^2(\Omega)} \lesssim \left(H^2 + e^{-c\ell} + \mathbf{wcba}(A,\tri_H)\right) \left(\|f\|_{L^2(\Omega)} + \|u^0\|_\X\right).
	\end{equation*}
\end{theorem}

\begin{proof}%
	We split the error $u - u_H = (u - \bar u_H) + (\bar u_H - \tilde u_H) + (\tilde u_H - u_H)$ with the solutions $\bar u_H$ and $\tilde u_H$ of the auxiliary problems
	\begin{equation*}\label{eq:PGlod}
	a(\RiH \bar u_H, (1-\cor)v_H) = (f,v_H) - a(\EbH u^0,(1-\cor)v_H)
	\end{equation*}
	and
	\begin{equation*}\label{eq:PGlodEll}
	a(\RiH \tilde u_H, (1-\cor_\ell)v_H) = (f_H,v_H) - a(\EbH u^0,(1-\cor_\ell)v_H).
	\end{equation*}
	To bound $e_H := u_H - \tilde u_H$, we observe that 
	\begin{equation}\label{eq:aId}
	a((1-\cor_\ell)e_H, (1-\cor_\ell)v_H) =  a(\cor_\ell \RiH \tilde u_H,(1-\cor_\ell)v_H) = a(\cor_\ell \RiH \tilde u_H, (\cor - \cor_\ell)v_H),
	\end{equation} 
	by the orthogonality property \eqref{eq:ortho}. 
	For the next steps, we require the identity
	\begin{equation}\label{eq:IHid}
	v_H = I_H (1-\cor_\ell) v_H,
	\end{equation}
	which follows from the fact that $\cor_\ell v_H \in W = \Ker I_H$ and thus $I_H\cor_\ell v_H = 0$.
	Testing with $v_H = e_H$ in \eqref{eq:aId} and using \eqref{e:diffcor1} as well as
	\begin{equation*}
	\|\nabla e_H\|_{L^2(\Omega)} = \|\nabla I_H (1-\cor_\ell)e_H\|_{L^2(\Omega)} \lesssim \|A^{1/2}\nabla (1-\cor_\ell)e_H\|_{L^2(\Omega)},
	\end{equation*}
	we obtain
	\begin{equation*}
	\begin{aligned}
	\|A^{1/2}\nabla(1-\cor_\ell)e_H\|^2_{L^2(\Omega)} &= a((1-\cor_\ell)e_H,(1-\cor_\ell)e_H) \\&= a(\cor_\ell \RiH \tilde u_H, (\cor - \cor_\ell)e_H) \\&\lesssim e^{-c\ell}\|\nabla \cor_\ell \RiH \tilde u_H\|_{L^2(\Omega)} \|A^{1/2}\nabla(1-\cor_\ell)e_H\|^2_{L^2(\Omega)}.
	\end{aligned}
	\end{equation*}
	Further, it follows that
	\begin{equation}\label{eq:errEH}
	\|e_H\|_{L^2(\Omega)} \lesssim \|A^{1/2}\nabla(1-\cor_\ell)e_H\|_{L^2(\Omega)} \lesssim e^{-c\ell}\left(\|f\|_{L^2(\Omega)} + \|u^0\|_\X\right)
	\end{equation}
	where we use \eqref{eq:IHid} and \eqref{eq:IH}.
	As a next step, we bound $\bar e_H := \tilde u_H - \bar u_H$. 
	We note that 
	\begin{equation*}
	a(\bar e_H, (1-\cor)v_H) = a(\RiH \tilde u_H + \EbH u^0, (\cor_\ell - \cor)v_H)
	\end{equation*}
	for any $v_H \in \VHo$. 
	With $v_H = \bar e_H$ and similar arguments as above, we obtain
	\begin{equation}\label{eq:errEHbar}
	\|\bar e_H\|_{L^2(\Omega)} \lesssim \|A^{1/2}\nabla(1-\cor)\bar e_H\|_{L^2(\Omega)} \lesssim e^{-c\ell}\left(\|f\|_{L^2(\Omega)} + \|u^0\|_\X\right).
	\end{equation}
	The error $u - \bar u_H$ can be estimated using \cite[Prop.~1]{GP17} which also holds for inhomogeneous Dirichlet boundary conditions, i.e.,
	\begin{equation}\label{eq:errE}
	\|u - \bar u_H\|_{L^2(\Omega)} \lesssim \left(H^2 + \mathbf{wcba}(A,\tri_H)\right) \left(\|f\|_{L^2(\Omega)} + \|u^0\|_\X\right).
	\end{equation}
	The triangle inequality, \eqref{eq:errEH}, \eqref{eq:errEHbar}, and \eqref{eq:errE} yield the desired estimate.
\end{proof}
To illustrate the advantage of the LOD, Figure~\ref{fig:FEMvsLOD} shows the error between the numerical solution on a microscopic scale and the numerical solutions using the LOD and a classical finite element approximation on a coarse scale, respectively.
The finite element method suffers from pre-asymptotic effects when the micro scale is not resolved, while the LOD produces a finite element function with much better approximation properties.
\begin{figure*}  
	\centering
	\captionsetup[subfigure]{labelformat=empty}
	\begin{subfigure}[b]{0.46\textwidth}  
		\centering 
		\includegraphics[width=\textwidth]{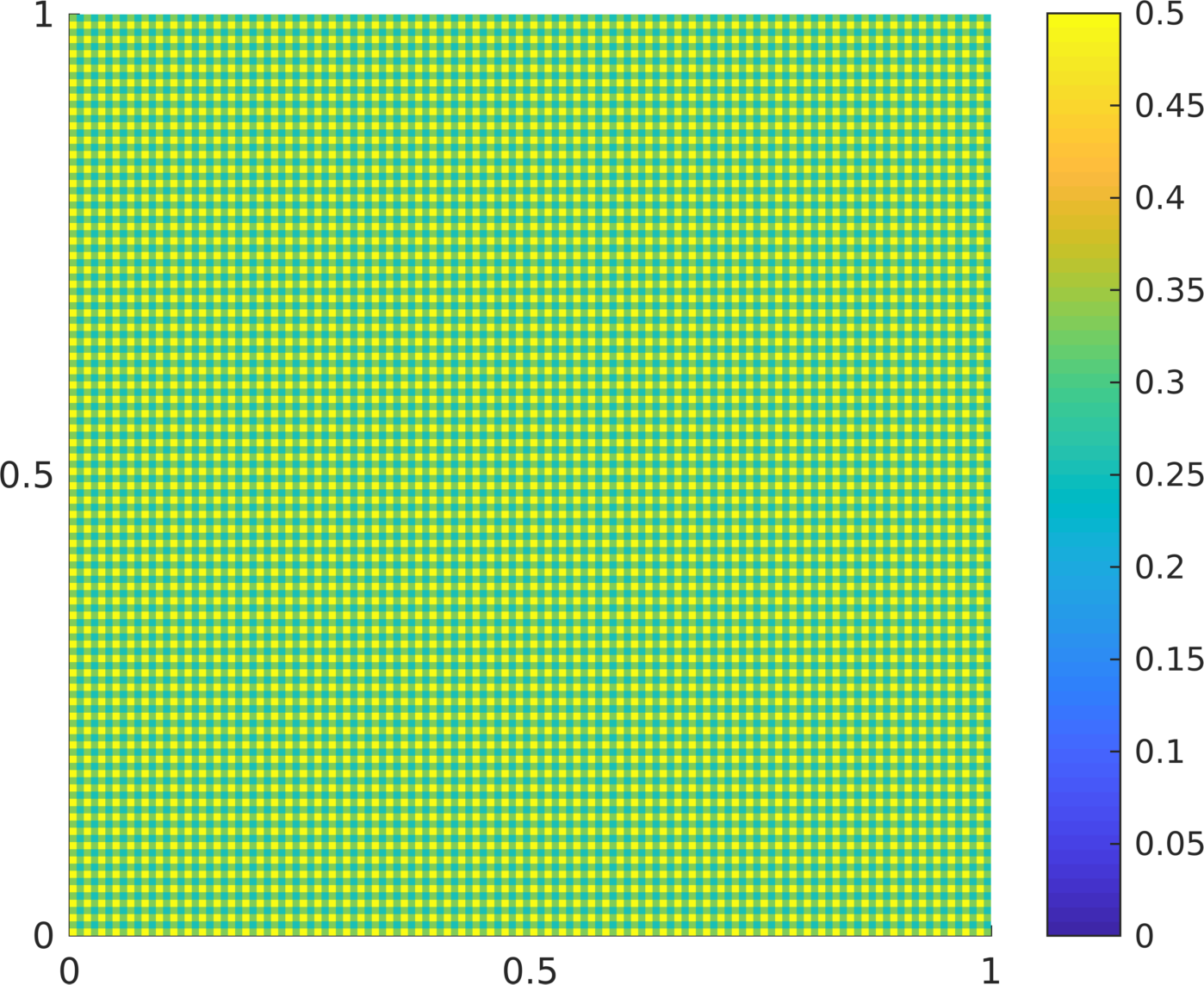}
		\caption[]%
		{\small Coefficient $A$}   
		\label{fig:coefficient}
	\end{subfigure}
	\quad
	\begin{subfigure}[b]{0.4\textwidth}
		\centering
		\includegraphics[width=\textwidth]{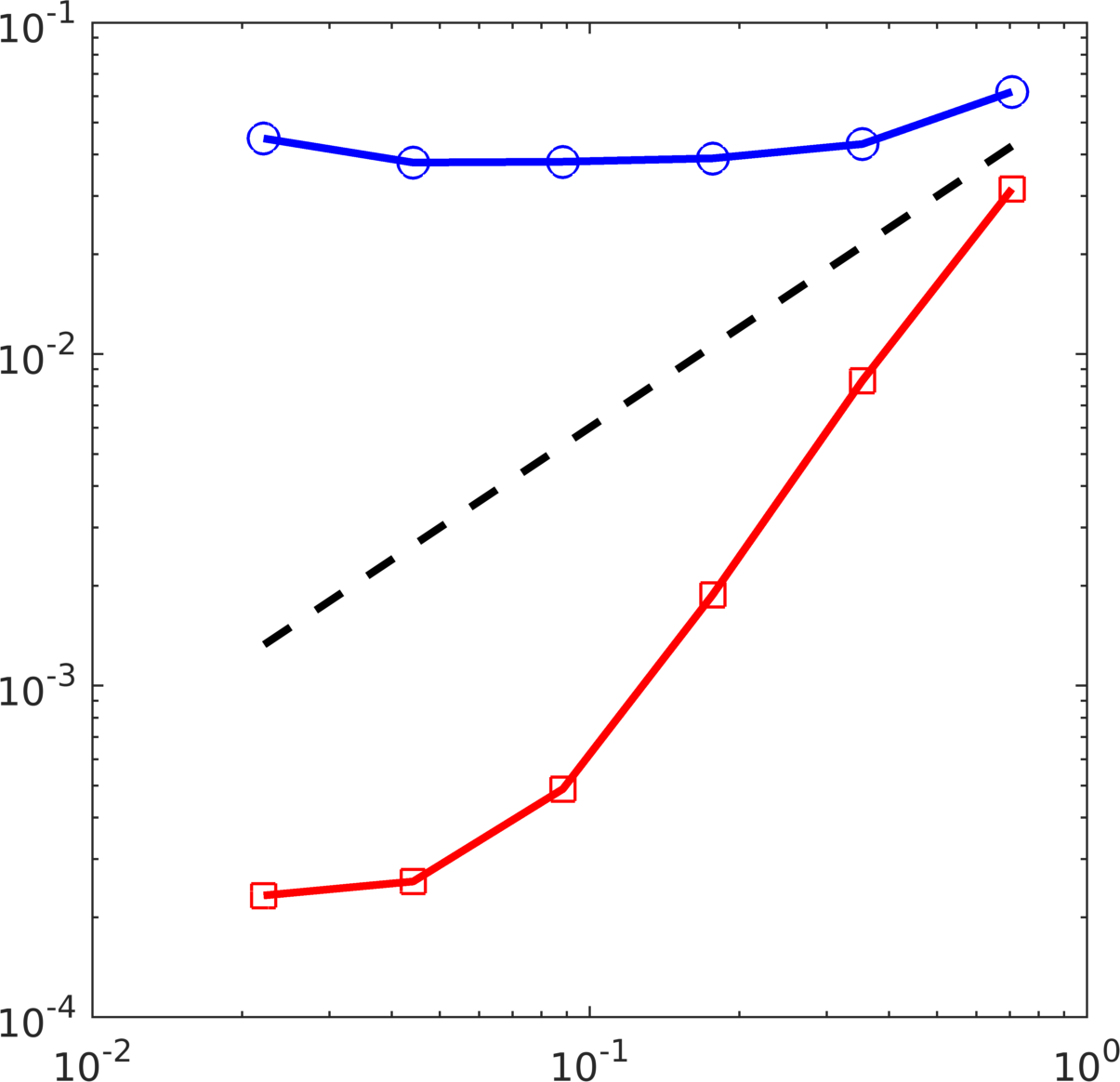}
		\caption
		{\small $L^2$ errors of FEM ({\protect \tikz{ \draw[line width=1.pt, blue] circle (0.6ex);}}) and LOD ({\protect \tikz{ \draw[line width=1.pt, red] (0,0) rectangle (0.18,0.18);}})}
		\label{fig:FEMvsLOD_err}    
	\end{subfigure}
	\caption[]
	{\small Left: An example of a microscopic coefficient. 
		Right: Comparison of the finite element method and the LOD on ${\Omega=[0,1]^2}$ for $f=1$, $u^0=0$, and $\ell=2$ for the solution of the diffusion problem corresponding to the depicted scalar coefficient. 
		The dashed line indicates linear convergence.}
	\label{fig:FEMvsLOD}
\end{figure*}

We emphasize that, choosing $\ell$ large enough (i.e., $\ell \gtrsim |\log H|$), it holds $e^{-c\ell} \lesssim H$ or even $e^{-c\ell} \lesssim H^2$. As discussed in \cite{GP17}, the worst-case best-approximation error is at least $\mathcal{O}(H)$, and it scales possibly even better with $H$ for certain pre-asymptotic regimes. 
In this work, we are mainly interested in solving the inverse problem and do not focus on optimizing the error estimates derived above. 

We can now define the discretized operator corresponding to \eqref{eq:discweakformH10} by
\begin{equation}\label{eq:opLell}
\begin{aligned}
\DtoN_{A,\ell}^\mathrm{eff}\colon \XH \times L^2(\Omega) & \to && \VH,\\
(u^{0},f) &\mapsto&& u_H, \text{ where } u_H \text{ solves } \eqref{eq:discweakformH10}.
\end{aligned} 
\end{equation}
Using Theorem~\ref{t:errForward} and the distance function defined in \eqref{eq:distf}, we obtain the following result.
\begin{corollary}[Error of the effective forward operator]\label{cor:errForward}
	Let $\ell \gtrsim |\log H|$. 
	Then it holds
	\begin{equation*}\label{eq:errOp}
	\dist_f(\DtoN_A,\DtoN_{A,\ell}^\mathrm{eff})\lesssim H.
	\end{equation*}
\end{corollary}

\subsection{Reformulation using the effective stiffness matrix}

As described in Section~\ref{ss:stiffmatrix}, we can alternatively represent the operator $\DtoN_{A,\ell}^\mathrm{eff}$ using the stiffness matrix corresponding to the discrete formulation \eqref{eq:discweakformH10}. 
Given a coefficient $A\in\AC$, the corresponding LOD stiffness matrix $S_H^\ell(A)$ is defined by
\begin{equation}\label{eq:S_HLOD} 
S_H^\ell(A)[i,j] := a((1-\cor_\ell)\Lambda_{z_j}, (1-\cor_\ell)\Lambda_{z_i}),\;
\mbox{for}\; i,\, j \in \{1,\dots,m\},
\end{equation}
with the ordering $i \mapsto z_i$ as in \eqref{eq:S_H}.
Further, the set of LOD stiffness matrices with oversampling parameter $\ell$ based on admissible coefficients is given by 
\begin{equation}\label{eq:defStiffness}
\SM(\ell,\tri_H) := \left\{S_H^\ell(A)\in \R^{m\times m}_\mathrm{sym}\,:\,A \in \AC \right\}.
\end{equation}
By construction of the correctors $\cor_\ell$ in \eqref{eq:corEll}, it holds \mbox{$\SM(\ell,\tri_H) \subseteq \MS(\ell,\tri_H)$} with the set $\MS(\ell,\tri_H)$ defined in Section~\ref{sec:forward}.  	

We can now prove the following lemma using the operator defined in \eqref{eq:opLSH}.

\begin{lemma}[Alternative representation of the effective forward operator]\label{l:LH=Lell}
	Let $S_H^\ell(A) \in \SM(\ell,\tri_H)$ be the LOD stiffness matrix corresponding to \eqref{eq:discweakformH10}. 
	Assume that $\Eb$ fulfills $\cor_\ell\EbH v^0 = \cor_\ell\Eb\vert_{\XH} v^0 = 0$ for any $v^0 \in \XH$. 
	Then it holds
	\begin{equation}\label{eq:LH=Lell} 
	\DtoN_{S_H^\ell(A)}(u^0,f) = \DtoN_{A,\ell}^\mathrm{eff}(u^0,f)
	\end{equation}
	for all $u^0\in\XH,\,f\in L^2(\Omega)$.
\end{lemma}
\begin{remark}
	Possible choices for the extension operator $\Eb$ that fulfill the assumptions of Lemma~\ref{l:LH=Lell} are those that extend functions in $\XH$ to functions in $\VH$ that are only supported on one layer of elements away from the boundary.
\end{remark}
\begin{proof}[Proof of Lemma~\ref{l:LH=Lell}]
	Write $u_H = \sum_{j = 1}^{m} u_j \Lambda_{z_j}$ and observe that \eqref{eq:discweakformH10} is equivalent to
	\begin{equation}\label{eq:discweakformLambda}
	\sum_{j\,:\,z_j \not\subset \partial\Omega}u_j\, a((1-\cor_\ell)\RiH \Lambda_{z_j}, (1-\cor_\ell)\Lambda_{z_i}) = (f_H,\Lambda_{z_i}) - a(\EbH u^0,(1-\cor_\ell)\Lambda_{z_i}) 
	\end{equation}
	for all $i \in \{k\,:\,z_k \not\subset \partial\Omega\}$. 
	Inserting $f_H = \sum_{j = 1}^{m} f_j \Lambda_{z_i}$, using the fact that 
	\begin{equation*}
	a(\EbH u^0,(1-\cor_\ell)v_H)=a((1-\cor_\ell)\EbH u^0,(1-\cor_\ell)v_H)
	\end{equation*}
	for any $v_H \in \VHo$, and the definition \eqref{eq:S_HLOD}, we can write equation \eqref{eq:discweakformLambda} as
	\begin{equation*} 
	S_{H,0}^\ell(A) \RiH u_H = R_H M_H f_H - \RiH S_H^\ell(A) \EbH u^0,
	\end{equation*}
	which shows \eqref{eq:LH=Lell}.
\end{proof}
Lemma~\ref{l:LH=Lell} and Corollary~\ref{cor:errForward} show that the operators $\DtoN_A(\cdot,f)$ and $\DtoN_{S_H^\ell(A)}(\cdot,f)$ are close as operators from $\XH$ to $\V$ if $\ell$ is chosen large enough. 
This property is essentially the message of Theorem~\ref{t:errEff} in Section~\ref{sec:forward}. In particular, the theorem holds with the explicit choice $\calG := 1-\cor_\ell$.

\section{Numerical Experiments}\label{sec:numerics}

In this section, we present some numerical experiments that illustrate the capability of the proposed method. 
The inverse problem is based on synthetic data, i.e., the coarse measurements used to feed the inversion algorithm are obtained from finite element functions in $\Vh$, defined on a mesh with mesh size $h = \sqrt{2}\cdot 2^{-9}$, that resolve the micro-scale features of the diffusion coefficient. 
Furthermore, the data are perturbed by multiplicative random noise with intensity up to 5\%.

\subsection{Example 1: full boundary data}\label{ss:fulldata}

In the first experiment, we assume to have full information on the operator (matrix) $\tilde\DtoN^\mathrm{eff}$, i.e., we assume that measurements in $\Omega$ on the coarse scale $H = \sqrt{2}\cdot2^{-5}$ for a complete basis of $\XH$ are available. 
The scalar coefficient $A$ for which the effective behavior should be recovered is constant on a mesh $\tri_\varepsilon$ with $\varepsilon = \sqrt{2}\cdot2^{-7}$ and the value on each element is independently obtained as a uniformly distributed random number between 1 and 50, i.e., for any $T \in \tri_\varepsilon$ we have $A|_{T} \sim U(1,50)$ (see Figure~\ref{fig:fulldata}, left, for the explicit sample used here). 
We set $f=1$ and start the inverse iteration with the finite element stiffness matrix $S_H^0$ based on the constant coefficient with value $1$. 
\begin{figure}[]
	\centering
	\includegraphics[width=.45\textwidth]{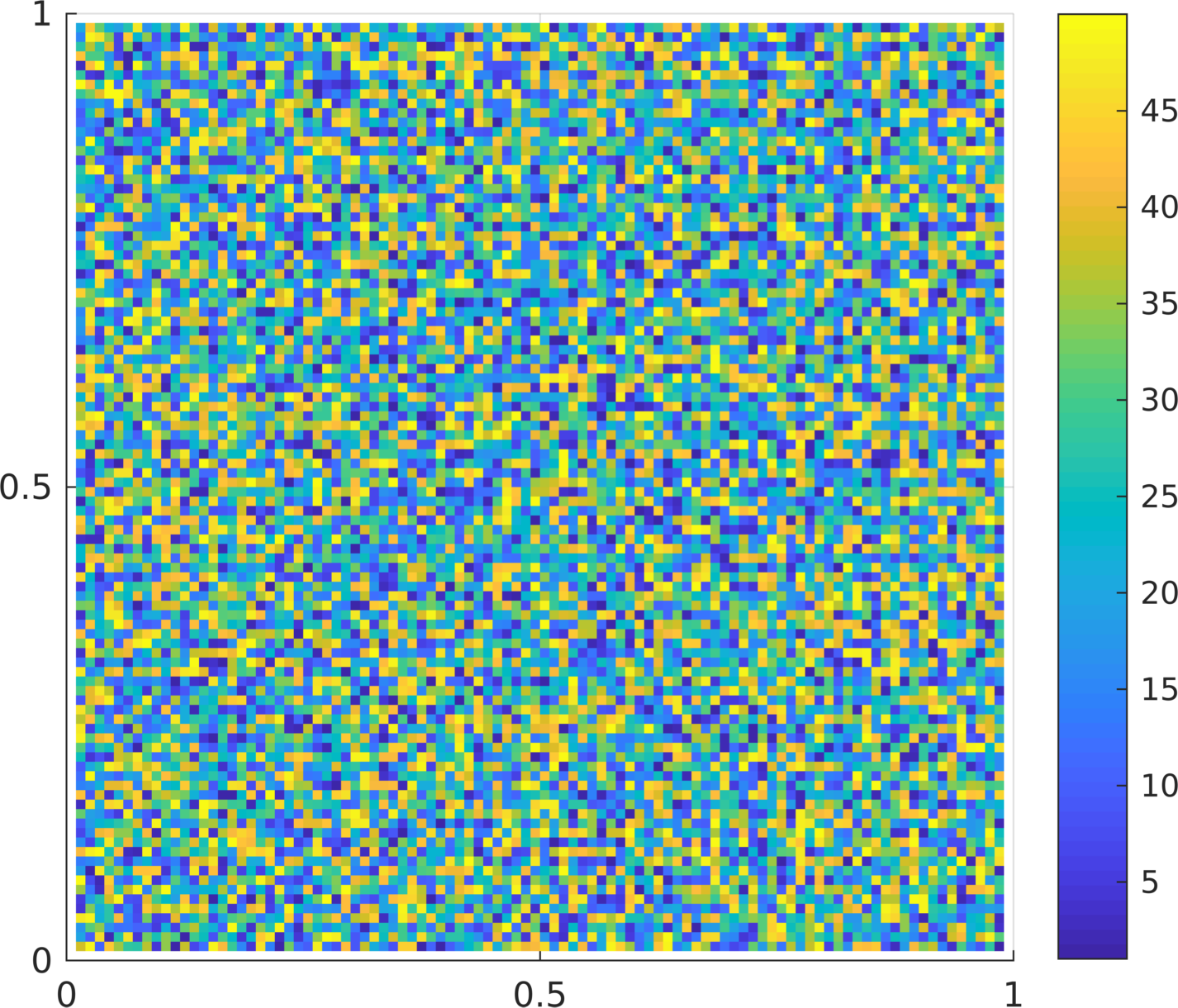}\hspace{0.2cm}
	\includegraphics[width=.52\textwidth]{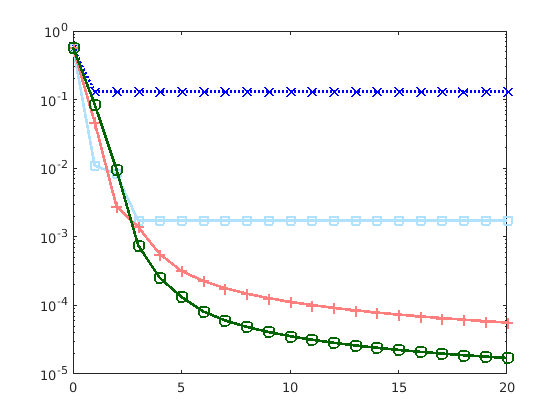}
	\caption[]
	{Left: Diffusion coefficient in Example~1. 
		Right: Values of $\mathcal{J}_H$ in the first 20 iterations of the inversion algorithm, using sparsity patterns based on local matrices ({\protect \tikz{ \draw node[thick, cross,  blue, fill]{};}}, dotted) and quasi-local matrices with $\ell = 1$ ({\protect \tikz{ \draw[line width=1.pt, color1] (0,0) rectangle (0.18,0.18);}}), $\ell = 2$ ({\protect \tikz{ \draw node[thick, cross, rotate=45, color2, fill]{};}}), $\ell = 3$ ({\protect \tikz{ \draw[line width=1.pt, color3] circle (0.6ex);}}).}
	\label{fig:fulldata}    
\end{figure}
The values of the functional $\mathcal{J}_H$ in the first 20 iterations of the inversion algorithm are given in Figure~\ref{fig:fulldata} (right). In particular, we compare the performance of a \emph{local approach} based on matrices in $\MS(0,\tri_H)$ with the sparsity pattern of, e.g., a standard first-order finite element method, the HMM, or the Two-Scale Finite Element Method, with a \emph{quasi-local approach} based on matrices in $\MS(\ell,\tri_H)$ for $\ell \in \{1,2,3\}$.
One clearly sees that slightly deviating from locality leads to better results in terms of decrease and value of the error functional $\mathcal{J}_H$. 
In particular, for $\ell = 0$ the functional seems to reach a stagnation relatively quickly, while the results significantly improve when increasing the value of $\ell$. 

A necessary validation step, in order to further investigate the behavior dependent on $\ell$, consists in solving a diffusion problem using the stiffness matrices reconstructed with using different sparsity patterns, and comparing the resulting numerical solutions with the finite element functions from which the measurements were taken to feed the inversion algorithm. 
The outcome of this assessment is shown in Figure \ref{fig:fulldata-crossX1}, focusing on the cross sections at $x_2 = 0.5$ (left) and at $x_1 = 0.5$ (right) of the numerical approximations corresponding to the boundary condition $u^0(x_1,x_2) = x_1$.
Figure \ref{fig:fulldata-crossX2} depicts the same cross sections when a random boundary condition $u^0\in\XH$ is considered.

\begin{figure}
	\centering
	\includegraphics[width=.48\textwidth]{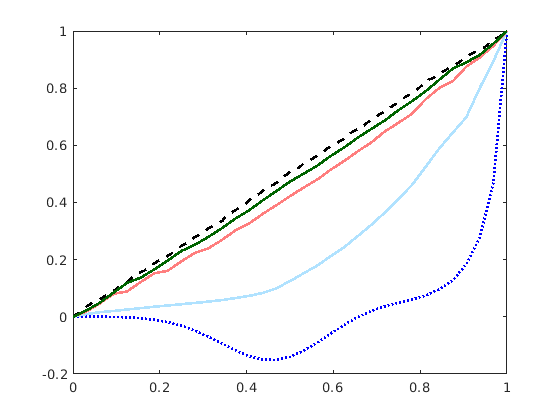}
	\hspace{0.3cm}
	\includegraphics[width=.48\textwidth]{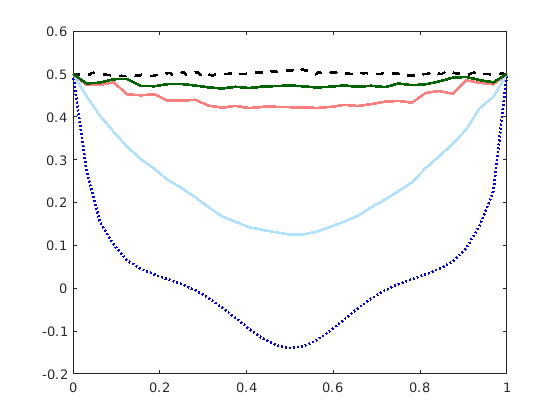}
	\caption[]
	{Cross sections of reconstructed functions with boundary condition $u^0(x_1,x_2) = x_1$ based on local stiffness matrices ({\protect \tikz{ \draw[line width=1.pt, blue, fill] circle (0.6ex);}}, dotted) and quasi-local ones with $\ell = 1$ ({\protect \tikz{ \draw[line width=1.pt, color1, fill] circle (0.6ex);}}), $\ell = 2$ ({\protect \tikz{ \draw[line width=1.pt, color2, fill] circle (0.6ex);}}), $\ell = 3$ ({\protect \tikz{ \draw[line width=1.pt, color3, fill] circle (0.6ex);}}) for Example~1 obtained from full boundary data. 
		The corresponding microscopic FE function 
		({\protect \tikz{ \draw[line width=1.pt, black, fill] circle (0.6ex);}}, dashed) is depicted as a reference. 
		Left: Cross section at $x_2 = 0.5$.
		Right: Cross section at $x_1 = 0.5$.}
	\label{fig:fulldata-crossX1}  
\end{figure}

\begin{figure}
	\centering
	\includegraphics[width=.48\textwidth]{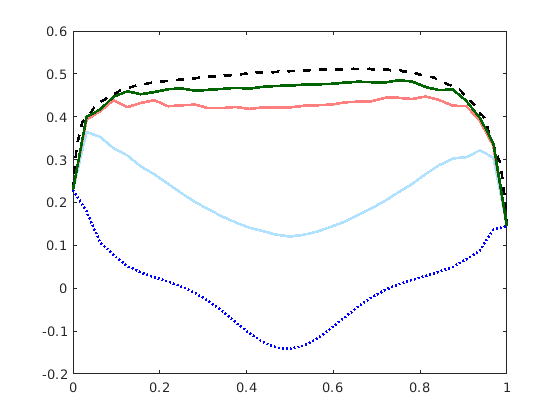}
	\hspace{0.3cm}
	\includegraphics[width=.48\textwidth]{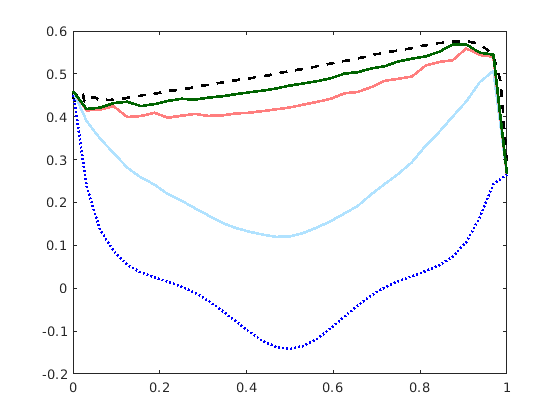}
	\caption[]
	{Cross sections of reconstructed functions with random boundary condition $u^0\in\XH$ based on local stiffness matrices ({\protect \tikz{ \draw[line width=1.pt, blue, fill] circle (0.6ex);}}, dotted) and quasi-local ones with $\ell = 1$ ({\protect \tikz{ \draw[line width=1.pt, color1, fill] circle (0.6ex);}}), $\ell = 2$ ({\protect \tikz{ \draw[line width=1.pt, color2, fill] circle (0.6ex);}}), $\ell = 3$ ({\protect \tikz{ \draw[line width=1.pt, color3, fill] circle (0.6ex);}}) for Example~1 obtained from full boundary data. 
		The corresponding microscopic FE function  
		({\protect \tikz{ \draw[line width=1.pt, black, fill] circle (0.6ex);}}, dashed) is depicted as reference. 
		Left: Cross section at $x_2 = 0.5$.
		Right: Cross section at $x_1 = 0.5$.}
	\label{fig:fulldata-crossX2}  
\end{figure}

Besides the accuracy of the numerical approximations computed based on the recovered stiffness matrices, it is also important to assess the robustness of the reconstructed effective model, i.e., to investigate to which extent the coarsened information about the diffusion coefficient encoded in the stiffness matrix can be used to simulate other scenarios.

For this purpose, we employ the reconstructed stiffness matrices to simulate a diffusion problem with two different right-hand sides, i.e., 
\begin{equation*}
g_1(x_1,x_2) = 20\,(\id_{\{x_1 < 0.5\}}\,x_1 + \id_{\{x_1 \geq 0.5\}}\,(1-x_1))(\id_{\{x_2 < 0.5\}}\,x_2 + \id_{\{x_2 \geq 0.5\}}\,(1-x_2))
\end{equation*}
and
\begin{equation*}
g_2(x_1,x_2) = 10\,\id_{\{x_1 \geq 0.5\}},
\end{equation*}
and compare the numerical results with the corresponding microscopic solution using the diffusion coefficient depicted in Figure \ref{fig:fulldata} (left).
In both cases, homogeneous Dirichlet boundary conditions are imposed on the outer boundaries.

Representative cross sections of the numerical approximations obtained based on the reconstructed stiffness matrices, compared to the corresponding microscopic solutions, are shown in Figure~\ref{fig:crossg}.
The numerical results indicate that robustness can be assured only with some moderate quasi-locality. 
Moreover, as in the previous experiments, the quality of the results improves if $\ell$ is increased. 

\begin{figure*}
	\centering
	\includegraphics[width=0.48\textwidth]{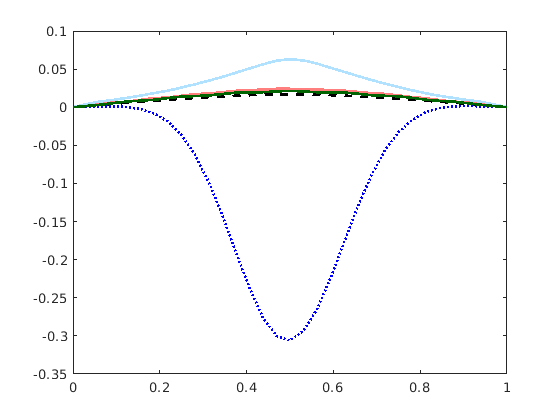}
	\hspace{0.3cm}
	\includegraphics[width=0.48\textwidth]{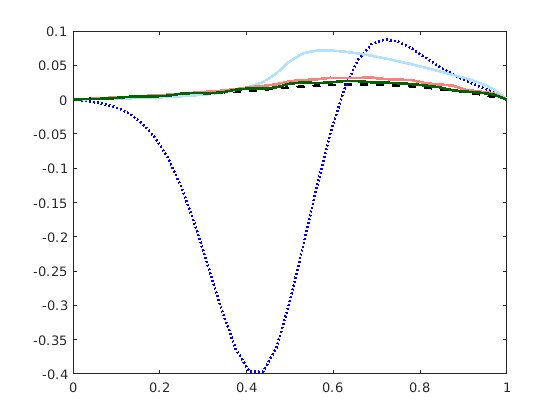}
	\caption[]
	{\small Cross sections at $x_2 = 0.5$ of reconstructed functions with homogeneous Dirichlet boundary conditions based on local stiffness matrices ({\protect \tikz{ \draw[line width=1.pt, blue, fill] circle (0.6ex);}}, dotted) and quasi-local ones with $\ell = 1$ ({\protect \tikz{ \draw[line width=1.pt, color1, fill] circle (0.6ex);}}), $\ell = 2$ ({\protect \tikz{ \draw[line width=1.pt, color2, fill] circle (0.6ex);}}), $\ell = 3$ ({\protect \tikz{ \draw[line width=1.pt, color3, fill] circle (0.6ex);}}). The corresponding microscopic FE functions ({\protect \tikz{ \draw[line width=1.pt, black, fill] circle (0.6ex);}}, dashed) are given as a  reference but were not part of the input data. 
		Left: Right-hand side $g_1$. 
		Right: Right-hand side $g_2$.}
	\label{fig:crossg}
\end{figure*}

\subsection{Example 2: incomplete boundary data}\label{ss:partialdata}

Next, we consider a more realistic case where the operator $\tilde\DtoN^\mathrm{eff}$ is only partially known. 
In practice, this means that coarse measurements in $\Omega$ are available only for $q$ distinct boundary conditions  in $\XH$ ($q < \mathrm{dim}\, \XH$). 
In this setting, the aim is to find an effective model that not only fits the given data, but that is also able to reproduce the coarse behavior for other boundary conditions not considered as input data. 

The scalar coefficient $A$ whose corresponding stiffness matrix should be recovered is shown in Figure~\ref{fig:partialdata} (left). 
We set $H=\sqrt{2}\cdot2^{-5}$, $f=1$, $q = 40$, and the initial matrix $S_H^0$ is defined as the finite element stiffness matrix based on an independent and uniformly distributed random coefficient on the coarse scale $H$ with values between $0.1$ and $10$.

We adapt the \emph{randomized approach} used in \cite{OY18} in the context of deep learning. 
Namely, in each iteration step, we randomly choose half of the available data to compute the new search direction, whereas we use all available data for the line search and for the evaluation of the functional $\mathcal{J}_H$. 
The values of the error functional $\mathcal{J}_H$ in the first 20 iterations of the inversion algorithm are shown in Figure~\ref{fig:partialdata} (right). 
One can observe that classical local stiffness matrices and even the quasi-local approach with $\ell=1$ cannot significantly improve the results obtained with the initial guess, while quasi-local matrices with $\ell \geq 2$ are able to reduce the values of the functional up to a certain degree.

As in the previous subsection, we validate the outcome of the inversion algorithm by solving a diffusion problem using the reconstructed stiffness matrices and comparing the numerical results with the corresponding microscopic finite element solutions. 
The cross sections at $x_2 = 0.5$ and $x_1 = 0.5$ of the numerical approximations using the different stiffness matrices are shown in Figure~\ref{fig:partialdata-cross1}, for the case with boundary condition $u^0(x_1,x_2) = x_1$.
We emphasize that, in this setting, neither the reference finite element function (black dotted line in Figure~\ref{fig:partialdata-cross1}) nor a coarse measurement from it were part of the input data.
\begin{figure}
	\centering
	\includegraphics[width=.45\textwidth]{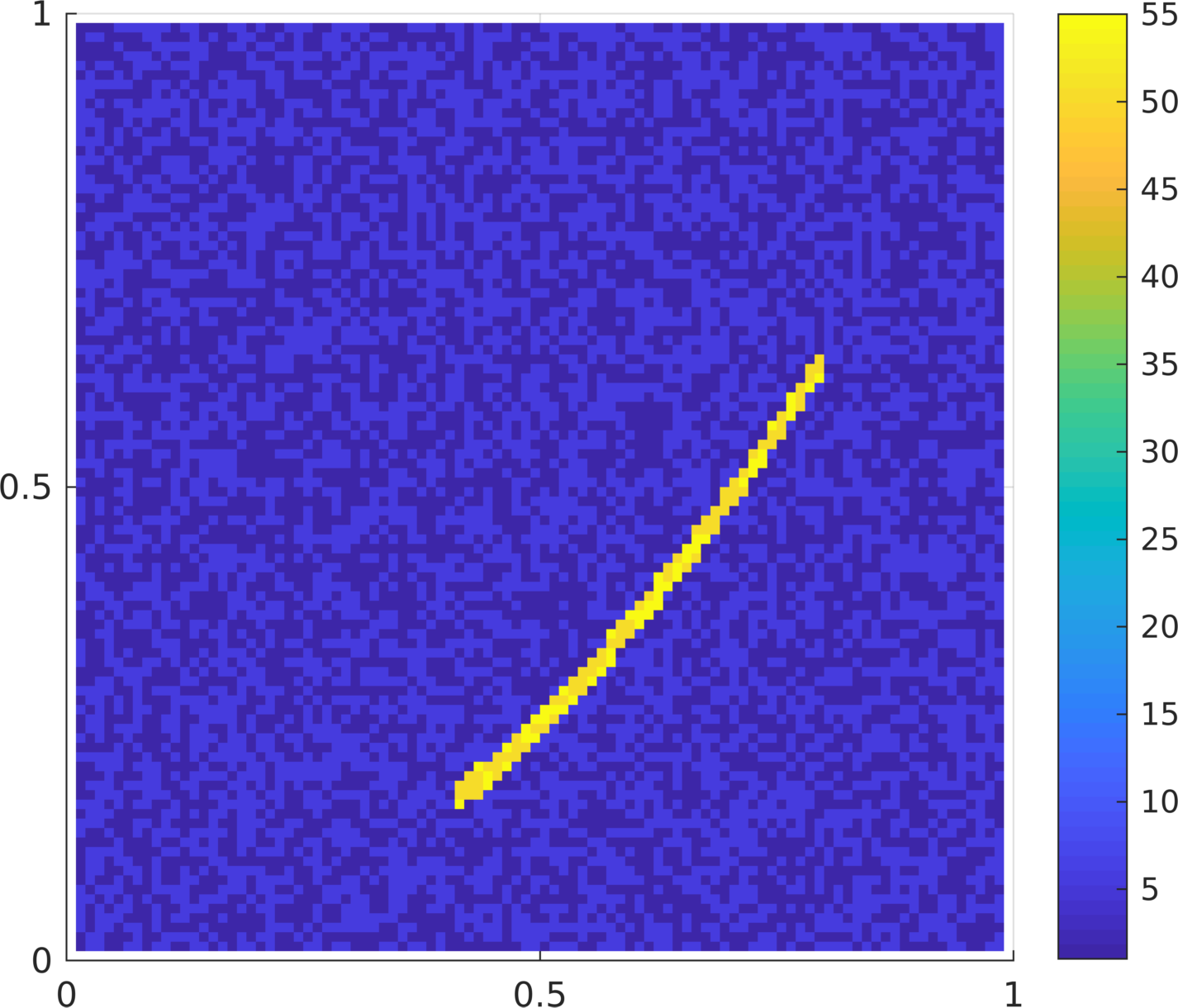}\hspace{0.2cm}
	\includegraphics[width=.52\textwidth]{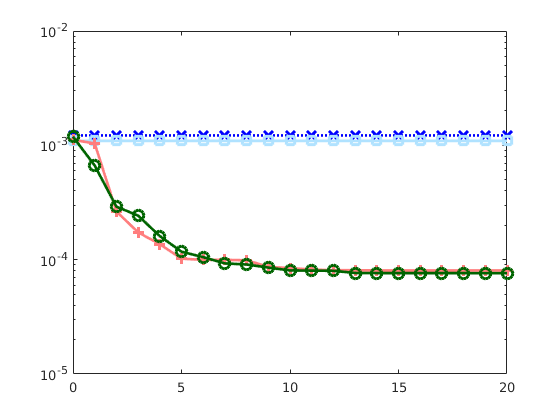}
	\caption[]
	{Left: Diffusion coefficient in Example~2. 
		Right: Values of $\mathcal{J}_H$ in the first 20 iterations of the inversion algorithm based on local matrices ({\protect \tikz{ \draw node[thick, cross,  blue, fill]{};}}, dotted) and quasi-local matrices with $\ell = 1$ ({\protect \tikz{ \draw[line width=1.pt, color1] (0,0) rectangle (0.18,0.18);}}), $\ell = 2$ ({\protect \tikz{ \draw node[thick, cross, rotate=45, color2, fill]{};}}), $\ell = 3$ ({\protect \tikz{ \draw[line width=1.pt, color3] circle (0.6ex);}}).}
	\label{fig:partialdata}    
\end{figure}
\begin{figure}
	\centering
	\includegraphics[width=.48\textwidth]{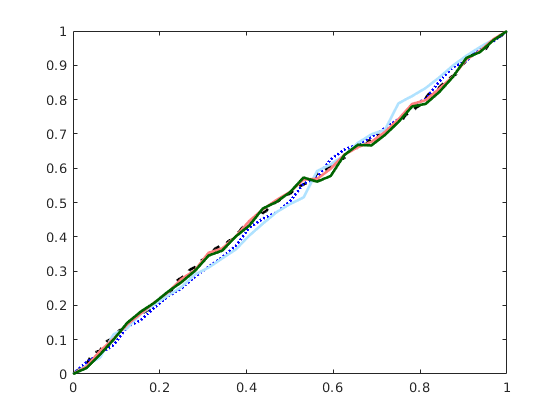}
	\hspace{0.3cm}
	\includegraphics[width=.48\textwidth]{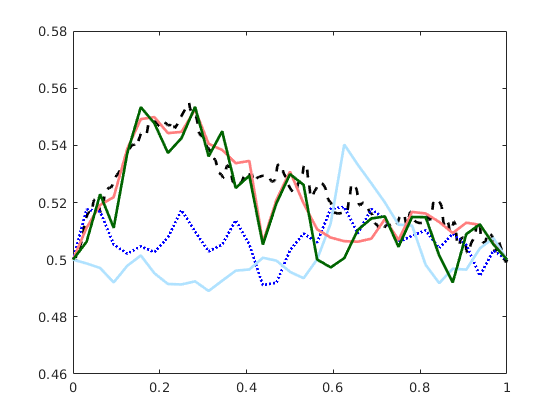}
	\caption[]
	{Cross sections of reconstructed functions with boundary condition $u^0(x_1,x_2) = x_1$ based on local stiffness matrices ({\protect \tikz{ \draw[line width=1.5pt, blue, fill] circle (0.6ex);}}, dotted) and quasi-local ones with $\ell = 1$ ({\protect \tikz{ \draw[line width=1.5pt, color1, fill] circle (0.6ex);}}), $\ell = 2$ ({\protect \tikz{ \draw[line width=1.pt, color2, fill] circle (0.6ex);}}), $\ell = 3$ ({\protect \tikz{ \draw[line width=1.pt, color3, fill] circle (0.6ex);}}) for Example~2 obtained from incomplete boundary data and the \textit{randomized approach}. 
		The corresponding microscopic FE function ({\protect \tikz{ \draw[line width=1.pt, black, fill] circle (0.6ex);}}, dashed) is depicted as a reference but was not part of the input data. 
		Left: Cross section at $x_2 = 0.5$.
		Right: Cross section at $x_1 = 0.5$.}
	\label{fig:partialdata-cross1}
\end{figure}

For a further comparison, we also present in Figure~\ref{fig:partialdata-fulldataapproach} the same cross sections of the numerical solutions obtained from the stiffness matrices using a \emph{full-data approach}, i.e., when all available data ($40$ measurements) are used in every step to compute the new search direction. 
The reconstructed matrices behave similarly to the ones obtained with the randomized approach. 
However, it is worth mentioning that the randomized strategy is generally more robust in the case of incomplete boundary data, and additionally requires less computational effort. 
For further numerical experiments, see also \cite[Ch.~4]{Mai20}.

\subsection{Discussion} 

The presented inversion results demonstrate that the reconstruction of the stiffness matrix assuming a local sparsity pattern with communication only between neighboring degrees of freedom does not allow us to capture effective features of the microscopic problem, while the reconstruction based on a quasi-local approach, especially with $\ell \geq 2$, is able to mimic the effective behavior.

Furthermore, some quasi-locality appears to allow for robustness with respect to different right-hand sides, a property which allows us to employ the reconstructed effective model for the simulation of other scenarios, assuming that the microscopic properties remain unchanged.

Our experiments also indicate that a lower bound on $\ell$ seems to be necessary similar to the forward setting, where $\ell$ needs to be increased for smaller values of $H$ ($\ell \gtrsim |\log H|$) to obtain improvements in the first place. 
In that sense, our findings also deviate from the numerical results in \cite{GGS12} which indicate that truly local numerical homogenization might always be possible.

\begin{figure}
	\centering
	\includegraphics[width=.48\textwidth]{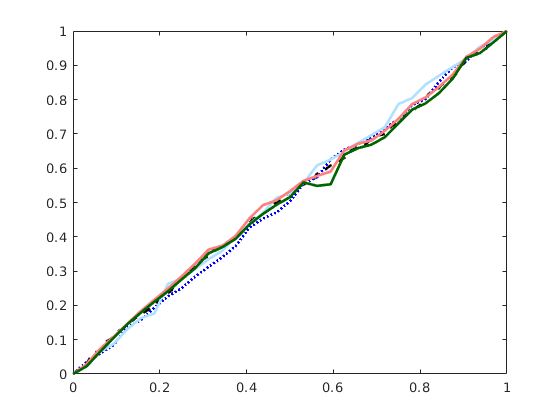}
	\hspace{0.3cm}
	\includegraphics[width=.48\textwidth]{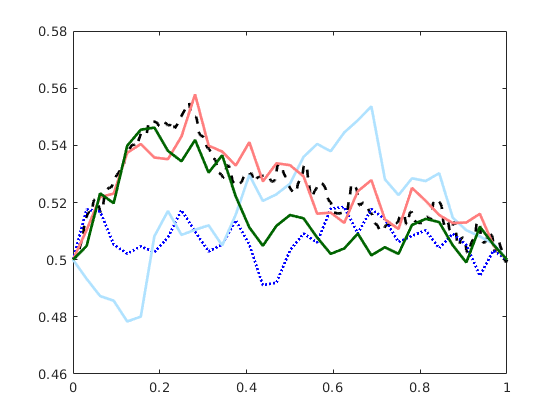}
	\caption[]
	{Cross sections of reconstructed functions with boundary condition $u^0(x_1,x_2) = x_1$ based on local stiffness matrices ({\protect \tikz{ \draw[line width=1.pt, blue, fill] circle (0.6ex);}}, dotted) and quasi-local ones with $\ell = 1$ ({\protect \tikz{ \draw[line width=1.pt, color1, fill] circle (0.6ex);}}), $\ell = 2$ ({\protect \tikz{ \draw[line width=1.pt, color2, fill] circle (0.6ex);}}), $\ell = 3$ ({\protect \tikz{ \draw[line width=1.pt, color3, fill] circle (0.6ex);}}) for Example~2 obtained from incomplete boundary data and the full-data approach. The corresponding microscopic FE function ({\protect \tikz{ \draw[line width=1.pt, black, fill] circle (0.6ex);}}, dashed) is depicted as a reference but was not part of the input data. 
		Left: Cross section at $x_2 = 0.5$.
		Right: Cross section at $x_1 = 0.5$.}
	\label{fig:partialdata-fulldataapproach}
\end{figure}

\section{Conclusion}\label{sec:conclusion}

We proposed a strategy to reconstruct the effective behavior of solutions of a multiscale PDE model which involves a coefficient varying on a microscopic scale. 
The approach is motivated by the effective models (represented by effective stiffness matrices) obtained by numerical homogenization. The aim is to provide a first step towards a reasonable surrogate in the inverse multiscale setting, which is characterized by a mismatch between coarse data scale and microscopic quantities.
The method relies on a quasi-local behavior in the sense that the reconstructed system matrices have a slightly denser sparsity pattern than standard finite element matrices, and this allows us to recover the behavior related to characteristic microscopic features of the solutions without requiring numerical computations on the microscopic scale. 
The method has been numerically validated on a prototypical model problem, considering a stationary linear elliptic diffusion problem with inhomogeneous boundary conditions.
Further, even the case of incomplete boundary data can be handled and ideas from learning-type methods may be adopted. 

A possible future extension of the approach includes the identification of further information about the underlying coefficient, e.g., geometric features, based on the reconstructed effective models.
Further, more involved combinations with learning-type techniques to deal with incomplete data could be studied, as well as adaptivity with respect to the parameter $\ell$.

\section*{Acknowledgments}
R.~Maier and D.~Peterseim gratefully acknowledge support by the German Research Foundation (DFG) in the Priority Program 1748 \emph{Reliable simulation techniques in solid mechanics. Development of non-standard discretization methods, mechanical and mathematical analysis} (PE2143/2-2).

\newcommand{\etalchar}[1]{$^{#1}$}

\end{document}